\documentclass[12pt]{article}

\usepackage[utf8]{inputenc}
\usepackage[english]{babel}
\usepackage{amssymb,amsmath}
\usepackage{amsthm}
\usepackage{float}
\usepackage{setspace,lipsum}
\usepackage{amsfonts}
\usepackage{graphicx}
\usepackage[margin=2.4cm]{geometry}
\usepackage{enumerate}
\usepackage{enumitem}
\usepackage{color}
\usepackage[noadjust]{cite}
\usepackage{tikz}
\usepackage{mathtools}
\usepackage{hyperref}
\usepackage{multirow}
\usepackage{array}
\usepackage[section]{placeins}
\usepackage{xfrac}
\usepackage{xspace}
\usepackage{ifthen}
\usepackage{longtable}

\newtheorem{thm}{Theorem}

\newtheorem{lem}[thm]{Lemma}

\newtheorem{defn}{Definition}

\title{Complete polyhedral description of chemical graphs of maximum degree at most 3}
\author{
	Valentin Dusollier\textsuperscript{1},
    S\'ebastien Bonte\textsuperscript{1},
	Gauvain Devillez\textsuperscript{1},
	\\
	Alain Hertz\textsuperscript{2},
	Hadrien M\'elot\textsuperscript{1},
	David Schindl\textsuperscript{3}	
	\\[3mm]
	\footnotesize \textsuperscript{1} Computer Science Department - Algorithms Lab\\[-2mm]
	\footnotesize University of Mons, Mons, Belgium\\[3mm]
	\footnotesize \textsuperscript{2} Department of Mathematics and Industrial
	Engineering\\[-2mm]
	\footnotesize Polytechnique Montr\'eal - Gerad, Montr\'eal, Canada\\[-2mm]
	\footnotesize Corresponding author, alain.hertz@gerad.ca\\[3mm]
	\footnotesize \textsuperscript{3} Haute Ecole de Gestion de Gen\`eve\\[-2mm]
	\footnotesize University of Applied Sciences Western Switzerland, Gen\`eve, Switzerland
}

\date{\today}

\newboolean{use_new_notation}
\setboolean{use_new_notation}{true}

\newcommand{\G}[1]{%
\ifthenelse{\equal{#1}{1}}{\ifthenelse{\boolean{use_new_notation}}{F6}{G1}}{%
\ifthenelse{\equal{#1}{2}}{\ifthenelse{\boolean{use_new_notation}}{F1}{G2}}{%
\ifthenelse{\equal{#1}{3}}{\ifthenelse{\boolean{use_new_notation}}{F2}{G3}}{%
\ifthenelse{\equal{#1}{4}}{\ifthenelse{\boolean{use_new_notation}}{F12}{G4}}{%
\ifthenelse{\equal{#1}{5}}{\ifthenelse{\boolean{use_new_notation}}{F3}{G5}}{%
\ifthenelse{\equal{#1}{6}}{\ifthenelse{\boolean{use_new_notation}}{F13}{G6}}{%
\ifthenelse{\equal{#1}{7}}{\ifthenelse{\boolean{use_new_notation}}{F15}{G7}}{%
\ifthenelse{\equal{#1}{7bis}}{\ifthenelse{\boolean{use_new_notation}}{F16}{G7bis}}{%
\ifthenelse{\equal{#1}{8}}{\ifthenelse{\boolean{use_new_notation}}{F5}{G8}}{%
\ifthenelse{\equal{#1}{9}}{\ifthenelse{\boolean{use_new_notation}}{F4}{G9}}{%
\ifthenelse{\equal{#1}{10}}{\ifthenelse{\boolean{use_new_notation}}{F11}{G10}}{%
\ifthenelse{\equal{#1}{11}}{\ifthenelse{\boolean{use_new_notation}}{F7}{G11}}{%
\ifthenelse{\equal{#1}{12}}{\ifthenelse{\boolean{use_new_notation}}{F10}{G12}}{%
\ifthenelse{\equal{#1}{13}}{\ifthenelse{\boolean{use_new_notation}}{F9}{G13}}{%
\ifthenelse{\equal{#1}{14}}{\ifthenelse{\boolean{use_new_notation}}{F8}{G14}}{%
\ifthenelse{\equal{#1}{15}}{\ifthenelse{\boolean{use_new_notation}}{F20}{G15}}{%
\ifthenelse{\equal{#1}{16}}{\ifthenelse{\boolean{use_new_notation}}{F21}{G16}}{%
\ifthenelse{\equal{#1}{17}}{\ifthenelse{\boolean{use_new_notation}}{F17}{G17}}{%
\ifthenelse{\equal{#1}{18}}{\ifthenelse{\boolean{use_new_notation}}{F14}{G18}}{%
\ifthenelse{\equal{#1}{19}}{\ifthenelse{\boolean{use_new_notation}}{F18}{G19}}{%
\ifthenelse{\equal{#1}{20}}{\ifthenelse{\boolean{use_new_notation}}{F19}{G20}}{%
\ifthenelse{\equal{#1}{21}}{\ifthenelse{\boolean{use_new_notation}}{NOTYETDEFINED}{G21}}{%
\textbf{G?}%
}}}}}}}}}}}}}}}}}}}}}}}

\newcommand{\V}[1]{%
\ifthenelse{\equal{#1}{1}}{\ifthenelse{\boolean{use_new_notation}}{\textup{V1}}{\textup{V1}}}{%
\ifthenelse{\equal{#1}{2}}{\ifthenelse{\boolean{use_new_notation}}{\textup{V2}}{\textup{V2}}}{%
\ifthenelse{\equal{#1}{3}}{\ifthenelse{\boolean{use_new_notation}}{\textup{V3}}{\textup{V3}}}{%
\ifthenelse{\equal{#1}{6}}{\ifthenelse{\boolean{use_new_notation}}{\textup{V4}}{\textup{V6}}}{%
\ifthenelse{\equal{#1}{7a}}{\ifthenelse{\boolean{use_new_notation}}{\textup{V5}}{\textup{V7a}}}{%
\ifthenelse{\equal{#1}{7b}}{\ifthenelse{\boolean{use_new_notation}}{\textup{V6}}{\textup{V7b}}}{%
\ifthenelse{\equal{#1}{7c}}{\ifthenelse{\boolean{use_new_notation}}{\textup{V7}}{\textup{V7c}}}{%
\ifthenelse{\equal{#1}{8c}}{\ifthenelse{\boolean{use_new_notation}}{\textup{V8}}{\textup{V8c}}}{%
\ifthenelse{\equal{#1}{8d}}{\ifthenelse{\boolean{use_new_notation}}{\textup{V9}}{\textup{V8d}}}{%
\ifthenelse{\equal{#1}{9a}}{\ifthenelse{\boolean{use_new_notation}}{\textup{V10}}{\textup{V9a}}}{%
\ifthenelse{\equal{#1}{9b}}{\ifthenelse{\boolean{use_new_notation}}{\textup{V11}}{\textup{V9b}}}{%
\ifthenelse{\equal{#1}{9c}}{\ifthenelse{\boolean{use_new_notation}}{\textup{V12}}{\textup{V9c}}}{%
\ifthenelse{\equal{#1}{10a}}{\ifthenelse{\boolean{use_new_notation}}{\textup{V13}}{\textup{V10a}}}{%
\ifthenelse{\equal{#1}{10b}}{\ifthenelse{\boolean{use_new_notation}}{\textup{V14}}{\textup{V10b}}}{%
\ifthenelse{\equal{#1}{10c}}{\ifthenelse{\boolean{use_new_notation}}{\textup{V15}}{\textup{V10c}}}{%
\ifthenelse{\equal{#1}{11a}}{\ifthenelse{\boolean{use_new_notation}}{\textup{V16}}{\textup{V11a}}}{%
\ifthenelse{\equal{#1}{11b}}{\ifthenelse{\boolean{use_new_notation}}{\textup{V17}}{\textup{V11b}}}{%
\ifthenelse{\equal{#1}{11c}}{\ifthenelse{\boolean{use_new_notation}}{\textup{V18}}{\textup{V11c}}}{%
\ifthenelse{\equal{#1}{12a}}{\ifthenelse{\boolean{use_new_notation}}{\textup{V19}}{\textup{V12a}}}{%
\ifthenelse{\equal{#1}{12b}}{\ifthenelse{\boolean{use_new_notation}}{\textup{V20}}{\textup{V12b}}}{%
\ifthenelse{\equal{#1}{12c}}{\ifthenelse{\boolean{use_new_notation}}{\textup{V21}}{\textup{V12c}}}{%
\textbf{V?}%
}}}}}}}}}}}}}}}}}}}}}}

\begin{document}

\maketitle

\begin{center}
\emph{In memory of Pierre Hansen.}
\end{center}

\hrule

\vspace*{0.3cm}

\small
\noindent
\textbf{Abstract.}

\emph{Chemical graphs are simple undirected connected graphs, where vertices represent atoms in a molecule and edges represent chemical bonds.
A degree-based topological
index is a molecular descriptor used to study specific physicochemical
properties of molecules. Such an index is computed from the sum of the weights of the edges of a chemical graph, each edge 
having a weight defined by a formula that depends only on the degrees of its endpoints.
Given any degree-based topological index and given two integers $n$ and $m$, we are interested in determining chemical graphs of order $n$ and size $m$ that maximize or minimize the index. Focusing on chemical graphs with maximum degree at most 3, we show that this reduces to determining the extreme points of a polytope that contains at most 10 facets. We also show that the number of extreme points is at most 16, which means that 
for any given $n$ and $m$, there are very few different classes of
extremal graphs, independently of the chosen degree-based topological index.}

\vspace*{0.2cm}

\vspace*{0.2cm}
\hrule

\normalsize

\section{Introduction}

Graph-theoretic descriptors, known as topological indices, play a fundamental role in mathematical chemistry by capturing structural properties of molecules and predicting their physicochemical behaviors~\cite{Codding98, Devillers00, Devillers96, Hansch95, Karelson00, Kier76, Kier86}.
As their name suggests, these indices depend on the structure of the molecules when modeled as a connected simple undirected graph, with the vertices representing the atoms and edges representing the chemical bonds. A common class of topological indices, called \textit{degree-based topological indices} are computed as the sum of the weights of the edges of a graph, each edge 
having a weight defined by a formula that depends only on the degrees of its endpoints. As a consequence, considering the partition of the set of edges where each block consists of edges having the same two endpoints degrees, a degree-based topological index can be viewed as a linear combination of the block sizes of this partition.

One of the first and most popular degree-based topological indices is the Randi\'c connectivity index \cite{Randic1975}, introduced in 1975. Since then, numerous such indices have been defined. They were categorized for the first time as degree-based in 2013 in \cite{G13}, where the author describes and compares the most important of them, in particular regarding their correlations with physico-chemical parameters. As he claims, "To use a mild expression,
today we have far too many such descriptors, and there seems to lack a firm criterion to stop or slow down their proliferation.” In light of their links with physico-chemical properties, it is a natural question to ask what are the minimal or maximal values these indices can take. These questions have given rise to a large number of publications \cite{das2011abc,zhang2016abc,AFRG22,henning2007albertson,hansen2005albertson,cui2021ag,carballosa2022ag,VPVFS21,che2016forgotten,yan2010ga,deng2018ga,zhong2012harmonic, DENG2013,elumalai2018, das2016,sedlar2015, falahati2017,ali2023maximum,cruz2021sombor,li2022extremal,cruz2021extremal,liu2021reduced,deng2021molecular,ali2019extremal,gutman2004first,nikolic2003zagreb,furtula2010augmented, ali2021augmented,li2008survey, swartz2022survey, hansen2009variable,GFE14,liu2020some, chu2020extremal,raza2020bounds,ali2020sddi,ghorbani2021sddi,hertz2025}. In some of these papers, the index is optimized over the set of all graphs or all connected graphs. However, since degree-based topological indices concern physico-chemical properties of molecules, it is natural to narrow their study to chemical graphs that we now define.

As pointed out by Patrick Fowler \cite{PF}, and as mentioned in \cite{article33}, two definitions of chemical graphs appropriate to different kinds of carbon framework can be found in
the literature. Chemical graphs can be regarded as the skeletons of saturated hydrocarbons (such
as alkanes), which implies that the maximum degree in such graphs is at most 4. 
If instead the interest is
in (unsaturated) conjugated systems, such as alkenes, polyenes, benzenoids \cite{kwun_m-polynomials_2018}, and fullerenes \cite{shigehalli_computation_2016, sharma_degree-based_2022}, then the chemical graphs that model such compounds have maximum degree at most 3, since  a conjugated carbon atom participates in at
most three single bonds.

In this paper, we restrict our attention to chemical graphs with maximum degree at most 3. The main practical outcome of this paper, is that for any degree-based topological index and any given values $n$ and $m$, we are able to give the maximum and the minimum values that the index can take over the set of all chemical graphs with $n$ vertices and $m$ edges (provided that such  graphs exist), together with a characterization of the set of graphs reaching these optimum values.

\subsection{Our approach}

For a graph $G$ and two integers $i,j$ with $1\le i\le j$, an $ij$-edge in $G$ is an edge with endpoints of degree $i$ and $j$, and we denote  $m_{ij}$  the number of $ij$-edges in $G$.  A degree-based topological index is a linear function of the form $\sum_{i\le j}c_{ij}m_{ij}$, where $c_{ij}$ is any real number. For example, the Randi\'c index has $c_{ij}=\frac{1}{\sqrt{ij}}$. Assuming that $G$ is connected, of order at least 3, and of maximum degree at most 3, only 5 of the $m_{ij}$ values can be strictly positive, namely $m_{12}$,  $m_{13}$, $m_{22}$, $m_{23}$, and $m_{33}$. Moreover,  as will be explained in Section \ref{sec:basic}, fixing $n$ and $m$ reduces the degrees of freedom by 2 and we therefore work in a 3-dimensional space spanned by three of the five variables of interest. We have chosen $m_{12}$, $m_{13}$ and $m_{33}$, and we say that a triplet $(m_{12},m_{13},m_{33})$ is realizable for $(n,m)$ if there is at least one connected graph of order $n$, size $m$, maximum degree at most 3, and with exactly $m_{12}$ $12$-edges, $m_{13}$ $13$-edges, and $m_{33}$ $33$-edges.

The linearity property of degree-based topological indices together with the aim of finding extremal values for them hinted us to carry out a polyhedral study of chemical graphs of maximum degree at most 3. Indeed, once we know how a polyhedron looks like, it is a trivial task to optimize a linear function of the underlying variables over that polyhedron.
In our analysis, we partition the set of all possible combinations of $n$ and $m$ into 96 cases. In each of these cases, we explicitly describe the polytope $\mathcal{P}_{n,m}$ given by the convex hull of all points $(m_{12}, m_{13}, m_{33})$ that are realizable for the corresponding pair $(n,m)$. 

As will be shown all these polytopes contain at most 10 facets and at most 16 vertices. Concretely, this means that for any given $n$ and $m$, there are very few different classes of extremal graphs, independently of the chosen degree-based topological index. This provides insight into why only a few graph families are sufficient to characterize the extremal graphs of many different degree-based topological indices, as was observed in \cite{article33}.

\subsection{Structure of the paper}

Our paper is organized as follows. In Section \ref{sec:basic}, we introduce some basic definitions and notations. In Section \ref{sec:Degenerated}, we explicit all cases where the polytopes $\mathcal{P}_{n,m}$ are degenerated, i.e. not 3-dimensional. In Section \ref{sec:valid}, we give a list of 21 inequalities on $m_{12}$, $m_{13}$ and $m_{33}$, each with associated conditions on $n$ and $m$. For each of these inequalities we show that they are valid for all chemical graphs of maximum degree at most 3 with $n$ vertices and $m$ edges, such that $n$ and $m$ satisfy the associated conditions. In Section \ref{sec:realizable}, we give a list of 21 points $(m_{12},m_{13},m_{33})$, each with its associated condition on $n$ and $m$ to be realizable for $(n,m)$. In Section \ref{sec:facet}, we show that the inequalities introduced in Section \ref{sec:valid}, subject to their respective conditions on $n$ and $m$, are actually facet defining and induce a polytope (i.e., a bounded polyhedron). In Section \ref{sec:extreme}, we partition the set of all possible combinations of $n$ and $m$ into 96 cases, and for each case, we explicitly give the set of facets and the set of extreme points that define the associated polytope $\mathcal{P}_{n,m}$.
In Section \ref{sec:application}, we illustrate by an example how to apply our characterizations to find the optimal value of a degree-based topological index. Finally, we discuss in Section \ref{sec:futurework} how to extend our results to chemical graphs of maximum degree at most 4.

\section{Basics and notations}\label{sec:basic}

As mentioned in the previous section, we denote $m_{ij}$ the number of $ij$-edges in $G$. Let $n_i$ be the number of vertices in $G$ of degree $i$. A \emph{chemical graph} of order $n$ and size $m$ is a connected graph with maximum degree at most 3. If $n \le 2$ then the cliques of order 1 and 2 are the only chemical graphs of that order. Thus, in the following, we only consider chemical graphs with at least 3 vertices, which gives
\begin{align} 2 \le n - 1 \le m \le \min\left\{\left\lfloor \frac{3n}{2} \right\rfloor, \frac{n(n-1)}{2}\right\}.\label{val_n_m}
\end{align}

Since chemical graphs are connected, Inequations (\ref{val_n_m}) imply $m_{11}=0$. Hence, a chemical graph can have at most 5 nonzero $m_{ij}$ values, namely $m_{12}, m_{13}, m_{22}, m_{23}$ and $m_{33}$. We therefore have:
\begin{align}
 \label{eq_n1}  n_1 & = m_{12} + m_{13}\\[-3pt]
 \label{eq_n2}  n_2 & = \frac{m_{12} + 2 m_{22} + m_{23} }{2}\\[-3pt] 
 \label{eq_n3}  n_3 & = \frac{m_{13} + m_{23} + 2 m_{33} }{3}\\[-3pt]
\label{n}n&=n_1+n_2+n_3=\frac{3}{2} m_{12} + \frac{4}{3} m_{13} + m_{22} + \frac{5}{6} m_{23} + \frac{2}{3} m_{33}\\[-3pt]
\label{m}m&=m_{12}+m_{13}+m_{22}+m_{23}+m_{33}.
\end{align}

Note that knowing $n$, $m$ and three of the five $m_{ij}$ values, we can deduce the other two. We have chosen to describe each chemical graph using $m_{12}$, $m_{13}$ and $m_{33}$. It follows from Equations (\ref{n}) and (\ref{m}) that the other two values are obtained as follows:
\begin{align}
\label{get_x22}
m_{22} &= 6n - 5m - 4 m_{12} - 3 m_{13} + m_{33}\\[-3pt]
\label{get_x23}
m_{23} &= 6m - 6n + 3 m_{12} + 2 m_{13} - 2 m_{33}.
\end{align}
This leads to the following definitions.
\begin{defn}
	A point $(m_{12}, m_{13}, m_{33})$ is \emph{realizable} for a pair $(n, m)$ if there exists a graph of order $n$ and size $m$ with $m_{12}$ 12-edges, $m_{13}$ 13-edges and $m_{33}$ 33-edges.
\end{defn}

\begin{defn}\label{def:polytope}
	Let $n$ and $m$ be two integers satisfying~\eqref{val_n_m}. The associated polytope $\mathcal{P}_{n, m}$ is the convex hull of all points $(m_{12},m_{13},m_{33})$ that are realizable  for $(n,m)$.  
\end{defn}

Note that a point of $\mathcal{P}_{n, m}$ possibly corresponds to several chemical graphs. For example, the two non-isomorphic chemical graphs of Figure \ref{figure2} both correspond to point $(1,0,0)$ in $\mathcal{P}_{6,6}$.

\begin{figure}[!htb]
	\centering\includegraphics[scale=0.75]{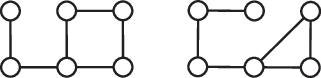}
	\vspace{-0.2cm}\caption{Two non-isomorphic chemical graphs with coordinates $(1,0,0)$ in $\mathcal{P}_{6,6}$}
	\label{figure2}
\end{figure}

In view of Definition \ref{def:polytope}, we consider that a full-dimensional polytope $\mathcal{P}_{n, m}$ is of dimension 3. Note that given a degree-based topological index $f$ with $c_{ij}$ values, we have:
\begin{align*}
    f(G)=&c_{12}m_{12}+c_{13}m_{13}+c_{22}m_{22}+c_{23}m_{23}+c_{33}m_{33}\\
    =&c_{12}m_{12}+c_{13}m_{13}+c_{22}(6n - 5m - 4 m_{12} - 3 m_{13} + m_{33})\\
    &+c_{23}(6m - 6n + 3 m_{12} + 2 m_{13} - 2 m_{33}) + c_{33} m_{33}\\
    =&(c_{12}-4c_{22}+3c_{23})m_{12}+(c_{13}-3c_{22}+2c_{23})m_{13}+(c_{22}-2c_{23}+c_{33})m_{33}\\
    &+(6n-5m)c_{22}+(6m-6n)c_{23}.
\end{align*}

Hence, maximizing or minimizing $f(G)$ over all chemical graphs of order $n$ and size $m$ is equivalent to maximizing or minimizing the linear function $c'_{12}m_{12}+c'_{13}m_{13}+c'_{33}m_{33}$ in $\mathcal{P}_{n,m}$, where \begin{itemize}[nosep]
\item $c'_{12}=c_{12}-4c_{22}+3c_{23}$,
\item $c'_{13}=c_{13}-3c_{22}+2c_{23}$,
\item $c'_{33}=c_{22}-2c_{23}+c_{33}.$
\end{itemize}

\section{Degenerated polytopes}\label{sec:Degenerated}

In this section, we describe degenerated polytopes, i.e., all polytopes of dimension $< 3$. As shown in Section \ref{sec:extreme}, the other polytopes are  full-dimensional. Some graphs mentioned in this section are shown in Figure \ref{figure6}.

\begin{figure}[!htb]	\centering\includegraphics[scale=0.73]{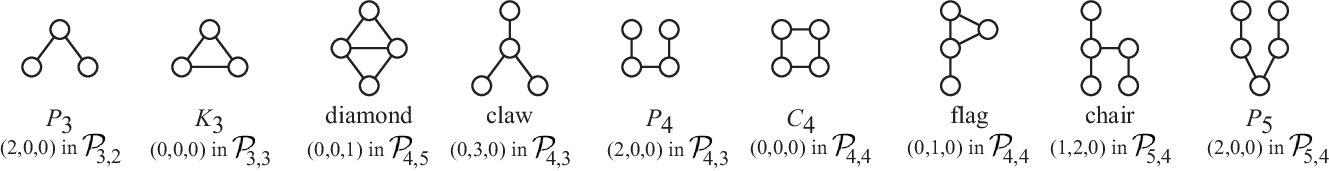}
	\vspace{-0.4cm}\caption{Some special graphs with their coordinates in a degenerated polytopes}
	\label{figure6}
\end{figure}

\subsection{0-dimensional polytopes}

$\mathcal{P}_{n,m}$ is a $0$-dimensional polytope, if it contains only one point. This is the case for $(n,m)=(3,2), (3,3)$ and $(4,5)$. Indeed, 
$P_3$ is the only point with coordinates $(2,0,0)$ in $\mathcal{P}_{3,2}$,
$K_3$ is the only point with coordinates $(0,0,0)$ in $\mathcal{P}_{3,3}$,
and the diamond is the only point with coordinates $(0,0,1)$ in $\mathcal{P}_{4,5}$.
Other 0-dimensional polytopes appear when $m = \lfloor \frac{3n}{2} \rfloor$:
\begin{itemize}[nosep]
	\item if $n$ is even, then the only point in $\mathcal{P}_{n,m}$ is $(0, 0, \frac{3n}{2})$, and all connected cubic graphs of order $n$ have these coordinates.
	\item  if $n$ is odd, then the only point in $\mathcal{P}_{n,m}$ is $(0, 0, \frac{3n-5}{2})$, and all connected graphs of order $n$ with one vertex of degree 2 and $n-1$ vertices of degree 3 have these coordinates. They can be obtained from a cubic graph with $n-1$ vertices by subdividing one of its edges into two edges.
\end{itemize}

\subsection{1-dimensional polytopes}

The polytope $\mathcal{P}_{n,m}$ has dimension $1$, if all its realizable points are aligned. In our case, all such polytopes contain exactly two realizable points. This happens when $(n,m)=(4,3),(4,4)$ or $(5,4)$. More precisely,
\begin{itemize}[nosep]
    \item if $(n,m) = (4,3)$, then
    \begin{itemize}[nosep]
        \item $(0, 3, 0)$ and $(2, 0, 0)$ are the only points in $\mathcal{P}_{4,3}$, and the only chemical graphs with these coordinates are the claw and $P_4$, respectively;
        \item $\mathcal{P}_{4,3}$ can be defined with the two equalities $m_{33} = 0$ and $3 m_{12} + 2 m_{13} = 6$, and the inequalities $0 \le m_{12} \le 2$;
    \end{itemize}
    \item if $(n,m) = (4,4)$, then
    \begin{itemize}[nosep]
        \item $(0, 1, 0)$ and $(0, 0, 0)$ are the only points in $\mathcal{P}_{4,4}$, and the only chemical graphs with these coordinates are the flag and $C_4$, respectively;
        \item $\mathcal{P}_{4,4}$ can be defined with the two equalities $m_{12} = 0$ and $m_{33} = 0$, and the inequalities $0 \le m_{13} \le 1$;
    \end{itemize}
    \item if $(n,m) = (5,4)$, then 
    \begin{itemize}[nosep]
        \item $(1, 2, 0)$ and $(2, 0, 0)$ are the only points in $\mathcal{P}_{5,4}$, and the only chemical graphs with these coordinates are the chair and $P_5$, respectively;
        \item $\mathcal{P}_{5,4}$ can be defined with the two equalities $m_{33} = 0$ and $2 m_{12} + m_{13} = 4$, and the inequalities $0 \le m_{13} \le 2.$
    \end{itemize}
\end{itemize}

\subsection{2-dimensional polytopes}

If $\mathcal{P}_{n,m}$ is a $2$-dimensional polytope, it can be described as the intersection of a plane with at least 3 facets defined by inequalities. This happens when $(n,m)=(5,6)$ or when $n \ge 6$ is even and 
$m = \frac{3n-2}{2}$.  The three points in $\mathcal{P}_{5,6}$ are shown in Figure \ref{figure3}, while when $n \ge 6$ is even and $m=\frac{3n-2}{2}$, the three points are:
\begin{itemize}[nosep]
    \item $(0,1,m-1)$, when we have one vertex of degree 1 and all others of degree 3. Such a graph is obtained from a cubic graph with $n-2$ vertices by subdividing an edge in two, and adding a pendant vertex to the newly created vertex;
    \item $(0,0,m-3)$, when we have 2 adjacent vertices of degree 2 and $n-2$ vertices of degree 3. Such a graph is obtained from a cubic graph with $n-2$ vertices by 
    subdividing an edge in three;
    \item $(0,0,m-4)$, when we have 2 non-adjacent vertices of degree 2 and $n-2$ vertices of degree 3. Such a graph is obtained from a cubic graph with $n-2$ vertices by subdividing two edges in two.
\end{itemize}

These 2-dimensional polytopes $\mathcal{P}_{n,m}$ are triangles defined by the equality $m_{12} = 0$ and by the three inequalities $m_{13} \ge 0$, $-3 m_{13} + m_{33} \ge 5m - 6n$, and $2 m_{13} - m_{33} \ge 6n - 5m - 1$. It is not difficult to check that the ends of the triangle are the points defined above, and are the only realizable points of $\mathcal{P}_{n,m}$.

\begin{figure}[!htb]
	\centering\includegraphics[scale=0.75]{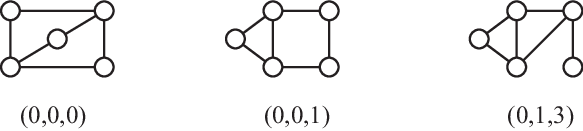}
	\vspace{-0.2cm}\caption{The three chemical graphs of order 5 and size 6, with their coordinates in  $\mathcal{P}_{5,6}$}
	\label{figure3}
\end{figure}

\section{Twenty-one valid inequalities }\label{sec:valid}

In view of the previous section, full-dimensional polytopes appear when $(n,m)=(5,5)$ and when $n\ge 6$ and $m \le  \lfloor\frac{3n-3}{2} \rfloor .$ We will analyze the full-dimensional polytopes with $m<12$ in Section \ref{sec:extreme}, and we therefore assume here that \begin{align}\max\{12,n-1\}\le m\leq\left\lfloor\frac{3n-3}{2} \right\rfloor. \label{m>12}\end{align}

Note that this implies $n \ge 9$. In this section we give 21 constraints that all have the form of an inequality that is imposed when certain conditions (in addition to condition (\ref{m>12})) on the order $n$ and the size $m$ of a chemical graph $G$ are satisfied. These constraints are shown in Table~\ref{constraints}. The first column indicates the Id of the constraint, the second column gives the inequalities, and the third one gives conditions on $n$ and $m$ (which have to be satisfied in addition to (\ref{m>12})).  Note that inequality \G{7bis} uses $m_{23}$ instead of $m_{33}$. This is because it is easier to read (there is no quadratic term) when compared to the inequality we would have by using $m_{12}$, $m_{13}$ and $m_{33}$ which is:

\begin{equation}\label{7bis}6 m^2 +  9 n^2 - 15 nm - 2 m + 3n \le (3n - 2m) (2m_{12} + m_{13} -m_{33})+2m_{13}.
\end{equation}

\begin{table}[h!]
	\centering
	\caption{Twenty-one constraints for chemical graphs satisfying (\ref{m>12})}\label{facets}~\\
		\begin{tabular}{|@{\hskip 1pt}c@{\hskip 1pt}|@{\hskip 1pt}r@{\hskip 1pt}c@{\hskip 1pt}l|c|}\hline\label{constraints}
			Id&\multicolumn{3}{c|}{Inequalities}&Conditions on $n$ and $m$\\\hline
		\G{2}		&		$0$		&$\le$&$ m_{12}$ 										&no condition\\
		\G{3}		&		$0$		&$\le$&$ m_{13}$ 										&no condition\\
		\G{5}		&$0$			&$\le$&$ m_{33}$   									&$m {\le} \frac{6n{-}4}{5}$\\
		\G{9}		&$2m{-}3n{+}2$	&$\le$&$ {-}3m_{12}$   								&$(m\bmod 3) {=} 2$ \\
		\G{8}		&$4m{-}6n{+}2$	&$\le$&${-}6 m_{12} {-} 3m_{13}$   					&$(m\bmod 3) {=} 1$ \\
		\G{1}		&	$5m {-} 6n$	&$\le$&$ {-}4 m_{12}{-}3m_{13}{+}m_{33}$				&no condition\\
		\G{11}		&$5m{-}6n{+}2$	&$\le$&${-}3m_{12}{-}3m_{13}{+}3m_{33}$				&$(m {\le} \frac{6n{-}5}{5})\land(m \bmod 3) {=} 2$\\
		\G{14}	&$5m{-}6n{+}2$	&$\le$&${-}4m_{12}{-}2m_{13}{+}2m_{33}$				&$(m {\le} \frac{6n-2}{5})\land(m {-} 2n \bmod 4) {=} 2$\\	
		\G{13}	&$5m{-}6n{+}3$	&$\le$&${-}4m_{12}{+}4m_{33}$ 							&$(m {\le} \frac{6n-3}{5})\land(m {-} 2n \bmod 4) {=} 1$\\
		\G{12}		&$15m{-}18n{+}3$&$\le$&${-}12m_{12}{-}8m_{13}{+}4m_{33}$				&$(m {\le} \frac{6n-1}{5})\land(m {-} 2n \bmod 4) {=} 3$\\
		\G{10}		&$10m{-}12n{+}2$&$\le$&${-}6m_{12}{-}6m_{13}{+}3m_{33}$				&$(m {\le} \frac{6n{-}4}{5})\land (m \bmod 3) {=} 1$\\
		\G{4}	&$3n {-} 3m$	&$\le$&$ m_{12} {+} m_{13} {-} m_{33}$					& $n{\le}m {\le} \frac{3n-6}{2}$ \\
		\G{6}		&$3n{-}3m{+}1$	&$\le$&$ 2m_{12}{+}m_{13}{-}m_{33}$ 					 &$(m \neq n) \land(n \bmod 2) {=} 1$\\
		\G{18}	&$0$			&$\le$&$ 2(n{-}2)m_{12}{+}(n{-}3)m_{13}{-}(n{-}1)m_{33}$ &$(m{=}n) \land(n \bmod 2) {=} 1$\\
		\G{7}		&${-}2n {+} 10$	&$\le$&$ 2(n{-}4)m_{12}{+}(n{-}2)m_{13}{-}(n{-}4)m_{33}$&$(m{=}n{+}1) \land(n \bmod 2) {=} 0$\\
		\G{7bis}	&$6n{-}4m$		&$\le$&$ (3n{-}2m)(m_{12}{+}m_{23}){+}4m_{13}$ 		&($m {\ge} n {+} 2) \land(n \bmod 2) {=} 0$\\
		\G{17}	&$4n{-}16$		&$\le$&$ 2(n{-}4) m_{12}{+}(n{-}4)m_{13}{-}(n{-}6)m_{33}$&$(m {=} n {-} 1) \land(n \bmod 2) {=} 0$\\
		\G{19}	&$2n{-}18$		&$\le$&$ (n{-}9)m_{12}{+}(n{-}6)m_{13}{-}(n{-}6)m_{33}$ &$(m {=} n {-} 1) \land(n \bmod 3) {=} 0$\\
		\G{20}		&$4n{-}32$		&$\le$&$ (2n{-}16)m_{12}{+}(2n{-}13)m_{13}{-}(2n{-}10)m_{33}$ &$(m{=}n{-}1) \land(n \bmod 3) = 2$ \\
		\G{15}	&$2n {-} 14$	&$\le$&$ (n{-}7)m_{12}{+}(n{-}6)m_{13}{-}(n{-}4)m_{33}$&$m {=} n {-} 1$\\
        		\G{16}		&$-1$			&$\le$&$ 2m_{12}{+}2m_{13}{-}m_{33}$					&$m{=}n{+}1$ \\\hline
	\end{tabular}
	\end{table}

The inequalities in Table \ref{facets} are named F$i$ $(i=1,\ldots,21)$, the reason for choosing the letter `F' being that, as we will see in Section \ref{sec:facet}, these inequalities are facet defining.
In this section, we only prove that they are all valid. In other words, a chemical graph of order $n$ and size $m$ necessarily satisfies inequality F$i$ if $n$ and $m$ satisfy the conditions associated with F$i$. Note that the inequalities F$i$ are possibly valid for other values of $n$ and $m$ but, as we will see in Section \ref{sec:facet}, the inequalities are then not facet defining.

The following theorem gives a proof that the 21 inequalities in Table \ref{facets} are valid. Before we give this proof, it is useful to note that 
Equations (\ref{eq_n1}), (\ref{eq_n3}), (\ref{n}) and (\ref{m}) imply
\begin{align}\label{n1n3}
n_1&=n_3+2(n-m).
\end{align}
Also, given any two integers $a$ and $b$, it directly follows from (\ref{n}) and (\ref{m}) that
\begin{align}\label{ab}
am+bn&=(a+\frac{3b}{2})m_{12}
+(a+\frac{4b}{3})m_{13}
+(a+b)m_{22}
+(a+\frac{5b}{6})m_{23}
+(a+\frac{2b}{3})m_{33}.
\end{align}
Note finally that $2m=3n-2n_1-n_2=3n-2m_{12}-2m_{13}-n_2$. Hence, 
\begin{align}m \equiv 2m_{13}+(n_2-m_{12})(\bmod 3)\label{mmod3}.
\end{align}

\begin{thm}
	If $n$ and $m$ satisfy condition (\ref{m>12}) as well as the conditions associated with a \emph{F}$i$ ($1\le i\le 21$), then all chemical graphs of order $n$ and size $m$ satisfy inequality \emph{F}$i$.
\end{thm}
\begin{proof}Let $n$ and $m$ be two integers satisfying condition (\ref{m>12}), and let $G$ be a chemical graph of order $n$ and size $m$, with $m_{ij}$ $ij$-edges.
		\begin{itemize}\setlength\itemsep{-1pt}
	\item \G{2}, \G{3} and \G{5} are clearly valid (even if $m>\frac{6n-4}{5}$ for \G{5}).
	\item If $m\bmod 3=2$, we know from (\ref{mmod3}) that $(2m_{13}+(n_2-m_{12}))\bmod 3=2$. Hence, $2m_{13}+(n_2-m_{12})\geq 2$ and \G{9} is valid since
	$$2m=3n-3m_{12}-(2m_{13}+(n_2-m_{12}))\leq 3n-3m_{12}-2.$$
		\item If $m\bmod 3=1$, we know from (\ref{mmod3}) that $m_{13}+2(n_2-m_{12})\bmod 3=2$. Hence, $m_{13}+2(n_2-m_{12})\geq 2$ and \G{8} is valid since (\ref{ab}) gives
		\begin{align*}
		6n-4m&=5m_{12}+4m_{13}+2m_{22}+m_{23}\\
		&=6m_{12}+3m_{13}+(m_{13}+2m_{22}+m_{23}-m_{12})\\
		&=6m_{12}+3m_{13}+(m_{13}+2(n_2-m_{12}))\\&\geq 6m_{12}+3m_{13}+2.
		\end{align*}
	\item \G{1} is valid since (\ref{ab}) gives
	$$
	5m-6n=-4m_{12}-3m_{13}-m_{22}+m_{33}\leq -4m_{12}-3m_{13}+m_{33}.
	$$
\item 
If $m \bmod 3=2$ then $m_{12}+m_{22}+2m_{33}\geq 2$. Indeed, this is obvious if $m_{33}\geq 1$ while for $m_{33}=0$, we have $m_{12}+m_{22}\geq 2$ since (\ref{get_x22}) gives $m_{13}=\frac{6n-5m-(4m_{12}+m_{22})}{3}$, which implies $m\bmod 3=(m_{12}+m_{22})\bmod 3=2$. Hence, \G{11} is valid since (\ref{ab})  gives   
\begin{align*}
5m-6n&=-4m_{12}-3m_{13}+m_{33}-m_{22}\\
&=-3m_{12}-3m_{13}+3m_{33}-(m_{12}+m_{22}+2m_{33})\\
&\leq-3m_{12}-3m_{13}+3m_{33}-2.	
\end{align*}		
\item If $(m-2n) \bmod 4=2$ then $m_{13}+m_{22}+m_{33}\geq 2$. Indeed, 3 times (\ref{get_x22}) plus 4 times (\ref{get_x23}) gives
$m_{23}=\frac{9m-6n-(m_{13}+3m_{22}+5m_{33})}{4}$, which implies $(m-2n)\bmod 4 = (m_{13}+3m_{22}+m_{33})\bmod 4=2$. Hence, \G{14} is valid since (\ref{ab}) gives
\begin{align*}
5m-6n&=-4m_{12}-3m_{13}+m_{33}-m_{22}\\
&=-4m_{12}-2m_{13}+2m_{33}-(m_{13}+m_{22}+m_{33})\\
&\leq-4m_{12}-2m_{13}+2m_{33}-2.
\end{align*}
\item If $(m-2n) \bmod 4=1$ then  $3m_{13}+3m_{33}+m_{22}\geq 3$. Indeed, this is obvious if $m_{13}+m_{33}\geq 1$ while for $m_{13}=m_{33}=0$ we have $m_{22}\geq 3$ since, as in the previous case, we deduce from  (\ref{get_x22}) and (\ref{get_x23}) that $m_{23}=\frac{9m-6n-3m_{22}}{4}$, which implies  $(m-2n)\bmod 4 = 3m_{22}\bmod 4=1$. Hence, \G{13} is valid since (\ref{ab})  gives
\begin{align*}
5m-6n&=-4m_{12}-3m_{13}+m_{33}-m_{22}\\
&=-4m_{12}+4m_{33}-(3m_{13}+3m_{33}+m_{22})\\
&\leq-4m_{12}+4m_{33}-3.
\end{align*}
\item If $(m-2n) \bmod 4=3$ then $m_{13}+m_{33}+3m_{22}\geq 3$. Indeed, this is obvious if $m_{22}\geq 1$ while for $m_{22}=0$ we have $m_{13}+m_{33}\geq 3$ since we deduce from (\ref{get_x22}) and (\ref{get_x23}) that $m_{23}=\frac{9m-6n-(m_{13}+5m_{33})}{4}$, which implies $(m-2n)\bmod 4 =(m_{13}+m_{33})\bmod 4=3$.
Hence, \G{12} is valid since (\ref{ab}) gives
\begin{align*}
15m-18n&=-12m_{12}-9m_{13}+3m_{33}-3m_{22}\\
&=-12m_{12}-8m_{13}+4m_{33}-(m_{13}+m_{33}+3m_{22})\\
&\leq-12m_{12}-8m_{13}+4m_{33}-3.
\end{align*}
\item If $m \bmod 3=1$ then $2m_{12}+2m_{22}+m_{33}\geq 2$. Indeed, this is obvious if $m_{12}+m_{22}\geq 1$ while for $m_{12}=m_{22}=0$ we must have $m_{33}\geq 2$ since (\ref{get_x22}) gives $m_{13}=\frac{6n-5m+m_{33}}{3}$, which implies $(m+m_{33}) \bmod 3 = (1+m_{33})\bmod 3 =0$. Hence, \G{10} is valid since (\ref{ab}) gives
\begin{align*}
10m-12n&=-8m_{12}-6m_{13}-2m_{22}+2m_{33}\\
&=-6m_{12}-6m_{13}+3m_{33}-(2m_{12}+2m_{22}+m_{33})\\
&\leq-6m_{12}-6m_{13}+3m_{33}-2.
\end{align*}
\item \G{4} is valid since $m>n-1$ implies that $G$ is not a path and we therefore have $m_{12}\leq m_{23}$, and (\ref{ab}) then gives
	$$3n{-}3m=\frac{3}{2}m_{12}{+}m_{13}{-}\frac{1}{2}m_{23}{-}m_{33}=m_{12}{+}m_{13}{-}m_{33}{+}\frac{m_{12}{-}m_{23}}{2}\leq m_{12}{+}m_{13}{-}m_{33}.
	$$
\item If $n$ is odd, then $G$ has at least one vertex of degree 2 and if in addition $m\neq n$,  $m_{12}+m_{23}\geq 2$ else $G$ is a cycle with $m=n$. This last inequality is equivalent to $\frac{m_{12}-m_{23}}{2}\leq m_{12}-1$. Hence, \G{6} is valid since (\ref{ab}) gives
	$$3n{-}3m=\frac{3}{2}m_{12}{+}m_{13}{-}\frac{1}{2}m_{23}{-}m_{33}=m_{12}{+}m_{13}{-}m_{33}{+}\frac{m_{12}{-}m_{23}}{2}\leq 2m_{12}{+}m_{13}{-}m_{33}-1.
		$$
\item If $n=m$, then we can assume that $G$ is not a cycle since \G{18} is clearly valid when $m_{12}=m_{13}=m_{33}=0$. Equation (\ref{n1n3}) gives $n_1=n_3$, which is equivalent to $m_{13}=\frac{n-n_2-2m_{12}}{2}$. We know from (\ref{ab}) that $3n-3m=0=\frac{3m_{12}-m_{23}}{2}+m_{13}-m_{33}$, which is equivalent to 
$2m_{12}+m_{13}-m_{33}=  \frac{m_{12}+m_{23}}{2}$ and $m_{12}-m_{33}=\frac{m_{23}-m_{12}}{2}-m_{13}=\frac{m_{12}+m_{23}+n_2-n}{2}$. Hence, proving \G{18} is equivalent to proving
\begin{align*}
0&\leq\frac{n-3}{2}(2m_{12}+m_{13}-m_{33})+m_{12}-m_{33}\\
&=\frac{n-3}{2}\left(\frac{m_{12}+m_{23}}{2}\right)+\frac{m_{23}+m_{12}+n_2-n}{2}.
\end{align*}
If $n=m$ is an odd number, then $n_2\geq 1$, which implies $m_{12}+m_{23}\geq 2$. Hence, $\frac{n-3}{2}\left(\frac{m_{12}+m_{23}}{2}\right)+\frac{m_{23}+m_{12}+n_2-n}{2}\geq \frac{n-3}{2}{-}\frac{n-3}{2}{=}0$, and \G{18} is therefore valid.

\item If $m=n+1$, we know from (\ref{ab}) that $3n-3m=-3=\frac{3m_{12}-m_{23}}{2}+m_{13}-m_{33}$, which is equivalent to 
 $2m_{12}+m_{13}-m_{33}+2=\frac{m_{12}+m_{23}}{2}-1$.
Proving the validity of \G{7} is therefore equivalent to proving
\begin{align*}
1&\leq \frac{n-4}{2}(2m_{12}+m_{13}-m_{33}+2)+m_{13}\\ &=\frac{n-4}{2}\left(\frac{m_{12}+m_{23}}{2}-1\right)+m_{13}.
\end{align*}
Note that $m_{12}+m_{23}$ is an even number. Hence, we have one of the five following cases:
\begin{itemize}[nosep]
	\item if $m_{12}+m_{23}\geq 4$, then the inequality is  clearly satisfied.
	\item if $m_{12}=m_{23}=1$, then $m_{13}\geq 1$, else $m_{33}=4$ and $n_3=3$ (by Equation \ref{eq_n3}), which is impossible. Hence $\frac{n-4}{2}\left(\frac{m_{12}+m_{23}}{2}-1\right)+m_{13}=m_{13}\geq 1$.
	\item if $m_{12}=0$ and $m_{23}=2$, then $m_{13}\geq 1$, else $m_{33}=2$ and $n_3=2$, which is impossible. Hence $\frac{n-4}{2}\left(\frac{m_{12}+m_{23}}{2}-1\right)+m_{13}=m_{13}\geq 1$.
    \item if $m_{12}=2$ and $m_{23}=0$, then $G$ is a path, a contradiction with $m=n+1$.
	\item if $m_{12}=m_{23}=0$, then $n_2=0$ and Equation (\ref{n1n3}) implies $m_{13}=\frac{n-2}{2}$. Hence $\frac{n-4}{2}\left(\frac{m_{12}+m_{23}}{2}-1\right)+m_{13}=-\frac{n-4}{2}+\frac{n-2}{2}= 1$.
\end{itemize}
\item To prove \G{7bis} we analyze the possible values of $m_{12}+m_{23}$ which is an even number.
\begin{itemize}\item If  $m_{12}+m_{23}\geq 2$ then  $(3n-2m) (m_{12}+m_{23}) +4 m_{13} \ge 2(3n-2m)=6n-4m.$
\item If $m_{12}+m_{23}=0$ then $m_{22}=0$ (since $m\neq n$). Hence $n_2=0$, which implies $m_{13}=\frac{3n-2m}{2}$. Therefore, $(3n-2m)(m_{12}+m_{23}) + 4m_{13}=2(3n-2m)=6n-4m$. 
\end{itemize}

\item If $m{=}n{-}1$, then we know from (\ref{ab}) that $6n-6m=6=3m_{12}+2m_{13}-m_{23}-2m_{33}$, which is equivalent to $2m_{12}+m_{13}-m_{33}-4=\frac{m_{12}+m_{23}}{2}-1$. Hence, proving \G{17} is equivalent to proving \begin{align*}
4&\leq\frac{n-6}{2}(2m_{12}+m_{13}-m_{33}-4)+2m_{12}+m_{13}\\
&=\frac{n-6}{2}\left(\frac{m_{12}+m_{23}}{2}-1\right)+2m_{12}+m_{13}.
\end{align*}
\begin{itemize}[nosep]
	\item if $m_{12}\geq 2$ then the inequality is clearly satisfied;
	\item if $m_{12}=1$, then $G$ is not a chain, which means that $m_{23}\geq 1$ and $n_1\geq 3$, which implies  $m_{13}\geq 2$. Hence, $\frac{n-6}{2}\left(\frac{m_{12}+m_{23}}{2}-1\right)+2m_{12}+m_{13}\geq 4$;
	\item if $m_{12}=0$ then $m_{23}$ is even and $m_{13}\geq 4$ (since $n>4$). Hence, the inequality is clearly satisfied if $m_{23}\geq 2$, while for $m_{23}{=}0$, we have $n_2{=}0$ and Equation (\ref{n1n3}) implies $m_{13}=\frac{n+2}{2}$, which gives
	$\frac{n-6}{2}\left(\frac{m_{12}+m_{23}}{2}-1\right)+2m_{12}+m_{13}=-\frac{n-6}{2}+\frac{n+2}{2}=4.$
\end{itemize}

\item If $m=n-1$, then we can assume $n_3>0$, else $m_{12}=2$ and $m_{13}=m_{33}=0$, which implies that inequality \G{19} is valid. We know from (\ref{ab}) that $3n-3m=3=\frac{3m_{12}-m_{23}}{2}+m_{13}-m_{33}$. Hence, $m_{12}+m_{13}-m_{33}=3+\frac{m_{23}-m_{12}}{2}$ which implies that $m_{33}\le m_{12}+m_{13}-3$. Therefore, proving \G{19} is equivalent to proving 
\begin{align*}
0&\leq(n-9)(m_{12}+m_{13}-m_{33}-2)+3(m_{13}-m_{33})\\
&=(n-9)\left(\frac{m_{23}-m_{12}}{2}+1\right)+3(m_{13}-m_{33}).
\end{align*}
We know from Equation (\ref{n1n3}) that $n_1{=}n_3{+}2$, which implies $\frac{3}{2}m_{12}{=}\frac{n+2-(n_2-m_{12})}{2}{-}m_{13}{\leq }\frac{n+2}{2}$. Hence, $m_{12}{\leq} \frac{n+2}{3}$, and assuming $n\bmod 3{=}0$ implies $m_{12}\leq \frac{n}{3}$. Since $m_{23}\geq m_{12}$ and $m_{13}-m_{33}\geq 3-m_{12}$, we have
\begin{align*}
(n-9)\left(\frac{m_{23}-m_{12}}{2}+1\right)+3(m_{13}-m_{33})&\geq (n-9)+3(3-m_{12})\\&\geq (n-9)+3(3-\frac{n}{3})=0.\end{align*}

\item If $m=n-1$, then as in the previous case, we know that $m_{12}+m_{13}-m_{33}=3+\frac{m_{23}-m_{12}}{2}$. Hence, proving \G{20} is equivalent to proving 
\begin{align*}
0&\leq(2n-16)(m_{12}+m_{13}-m_{33}-2)+3(m_{13}-2m_{33})\\
&=(2n-16)\left(\frac{m_{23}-m_{12}}{2}+1\right)+3m_{13}-6m_{33}.
\end{align*}
Let  $f(G)=(2n-16)\left(\frac{m_{23}-m_{12}}{2}+1\right)+3(m_{13}-2m_{33})$. We have to prove that if $G$ is a tree of order $n$ with $(n\bmod 3)=2$, then $f(G)\geq 0$.
We may assume $m_{33}>0$ and $n_3>0$ since $f(G)\geq 0$ when $n_3=0$ or $m_{33}=0$. 
So assume $G$ minimizes $f$ over  all trees of order $n$ with $(n\bmod 3)=2$, $n_3>0$ and $m_{33}>0$. It remains to prove that $f(G)\ge 0$.

Note first that $G$ does not contain a chain $[v_1,v_2,\ldots,v_p]$ ($p \ge 3$) with $v_1$ and $v_p$ of degree 3 and the other vertices of degree 2. Indeed if such a chain exists then let $uw$ be an edge in $G$ with $u$ of degree 1 and let $G'$ be the tree obtained from $G$ by removing $uw$, $v_1v_2$, $v_{p-1}v_p$ and adding $uv_2$, $wv_{p-1}$, $v_1v_p$. Then $f(G')=f(G)-(2n-16)-6<f(G)$ if $w$ has degree 2, and $f(G')=f(G)-(2n-16)-9<f(G)$ if $w$ has degree 3. Both cases contradict the fact that $G$ minimizes $f$.

So $m_{12}=m_{23}$, which implies $n_1=m_{12}+m_{13}=m_{33}+3$. Hence, $m_{22}=m-(m_{12}+m_{13})-m_{23}-m_{33}=m-6-3m_{33}+m_{13}$, which implies $3m_{33}=n-7-m_{22}+m_{13}$. Note also that since $n=3m_{33}+7+m_{22}-m_{13}$ and $n\bmod 3=2$, we have $(m_{22}-m_{13})\bmod 3=1$. Hence, $m_{13}+2m_{22}\ge 2$. Indeed, this is clearly true if $m_{22}\ge 1$, while $m_{22}=0$ implies $m_{13}\ge 2$. Therefore
\begin{align*}
f(G)&=2n-16+3m_{13}-2(n-7-m_{22}+m_{13})\\
&=m_{13}+2m_{22}-2\ge 0.
\end{align*}

\item Let $f(G)= (n-7)(m_{12}+m_{13}-m_{33}-2)+m_{13}-3m_{33}$. Proving the validity of \G{15} is equivalent to proving that $f(G)\geq 0$ for all trees $G$. If $G$ is a path, then $f(G)=0$. So let $G$ be a tree with $n_3>0$ and assume that $G$ has minimum value among all trees having at least one vertex of degree 3.

We first prove that $G$ does not contain any chain $[v_1,v_2,\ldots,v_p]$ with $v_1$ and $v_p$
of degree 3 and all other vertices of degree 2.
Consider any vertex $u$ of degree 1, let $w$ be its neighbor, and  let $G'$ be the tree obtained from $G$ by removing the edges $uw$, $v_1v_2$, $v_{p-1}v_p$ and adding the edges $uv_2$, $wv_{p-1}$, $v_1v_p$. Then $f(G'){=}f(G){-}(n{-}7){-}4$ if $w$ has degree 3, while $f(G'){=}f(G){-}(n{-}7){-}3$ if $w$ has degree 2. Hence, $f(G'){<}f(G)$, which contradicts the fact that $G$ minimizes $f$.

So $m_{12}=m_{23}$ which, as shown in the previous case, implies $m_{33}{=}m_{12}{+}m_{13}{-}3$ and $3m_{33}=n{-}7{-}m_{22}{+}m_{13}$. It follows that $f(G){=}(n{-}7){+}m_{13}{-}(n{-}7{-}m_{22}{+}m_{13}){=}m_{22}{\ge} 0$.

\item Let $f(G)=2m_{12}+2m_{13}-m_{33}$. To prove \G{16}, we have to prove that $f(G)\ge-1$.
Let $\cal{G}$ be the set of graphs that minimize $f(G)$ among all graphs of size $m=n+1$ that satisfy condition (\ref{m>12}) and let $G$ be a graph in $\cal{G}$ having the smallest number of vertices of degree 1. Since $m=n+1$, $G$ has a least one vertex of degree 3.

If $G$ has at least one vertex of degree 1, then let $[v_1,v_2,\ldots,v_p]$ ($p\ge 2$) be a chain in $G$ with $v_1$ of degree 1, $v_p$ of degree 3 and all other vertices of degree 2. Let $u\neq v_{p-1}$ be a second neighbor of $v_p$ and let $G'$ be the graph obtained from $G$ by replacing $uv_p$ by $uv_1$. Then $G'$ is still connected, has $m_{12}{+}m_{13}{-}1$ vertices of degree 1 and at least $m_{33}{-}2$ $33$-edges. Hence, $G'$ has less vertices of degree 1 than $G$ and $f(G'){\leq} f(G)$, a contradiction.

So $n_1{=}0$, and it follows from Equation (\ref{n1n3}) that $n_3{=}2$. Therefore, $f(G){=}{-}m_{33}{\geq}{-}1$.\qedhere
\end{itemize}
\end{proof}

\section{Twenty-one realizable points}\label{sec:realizable}

In this section, we give 21 points $(m_{12},m_{13},m_{33})$ and conditions on $n$ and $m$ so that they are realizable for $(n,m)$. These points appear in Table \ref{vi}, where the first column indicates the Id of the point and the last three columns give the 3 components $m_{12}$, $m_{13}$ and $m_{33}$. 

The conditions on $n$ and $m$ for realizability are given in the second column of Table \ref{ConditionsVi}.
As we will see in Section \ref{sec:extreme}, given any two integers $n$ and $m$ that satisfy condition (\ref{m>12}), the set of extreme points of the polytope $\mathcal{P}_{n,m}$ is a subset of these 21 points.
It may happen that a V$i$ is realizable for $(n,m)$ but not an extreme point of $\mathcal{P}_{n, m}$. The analysis in Section \ref{sec:extreme} will show that the conditions on $n$ and $m$ for V$i$ to be an extreme point of $\mathcal{P}_{n, m}$ are those shown in the third column  of Table  \ref{ConditionsVi}. So, for example, \V{12a} is realizable for $(n,m)$ with $m\le n$, while it is an extreme point of $\mathcal{P}_{n,m}$ only if $m=n-1$.

\begin{table}[h!]
	\centering
	\caption{Twenty-one realizable points}\label{vi}~\\[2ex]
	\begin{tabular}{|c|c|c|c|}\hline		
		Id&$m_{12}$&$m_{13}$&$m_{33}$\\\hline	
		\V{1} &$0$&$0$&$0$\\	
		\V{2}&$2$&$0$&$0$\\
		\V{3}&$0$&$0$&$1$ \\
		\V{6}&$0$&$0$&$5m - 6n$ \\
		\V{7a}&$\frac{ 6n -5m - 3((m-2n) \bmod 4)}{4}$&$(m-2n) \bmod 4$&$ 0$\\
		\V{7b}&$\frac{6n-5m + (m-2n) \bmod 4}{4}$&$0$&$(m-2n) \bmod 4$\\
		\V{7c}&$\frac{6n-5m - (2n-m) \bmod 4}{4}$&$0$&$0$\\
		\V{8c}&$0$&$\frac{ 3n-2m - n \bmod 2}{2}$&$\frac{4m - 3n - 3(n \bmod 2)}{2}$\\
		\V{8d}&$1$&$\frac{3n-2m - 3}{2}$&$\frac{4m - 3n - 1}{2}$\\
		\V{9a}&$3m - 3n - 2$&$3m - 3n - 2$&$6m - 6n - 1$\\
		\V{9b}&$3m - 3n - 1$&$0$&$6m - 6n - 1$\\
		\V{9c}&$0$&$3m - 3n - 2$&$6m - 6n - 3$\\
		\V{10a}&$0$&$\frac{6n-5m - m \bmod 3}{3}$&$0$ \\
		\V{10b}&$m \bmod 3$&$\frac{6n-5m - 4(m \bmod 3)}{3}$&$0$ \\
		\V{10c}&$0$&$\frac{6n-5m + (2m) \bmod 3}{3}$&$(2m) \bmod 3$\\
		\V{11a}&$\frac{3n-2m - m \bmod 3}{3}$&$0$&$\frac{7m - 6n - 4(m \bmod 3)}{3}$\\
		\V{11b}&$\frac{3n-2m - 2((2m) \bmod 3)}{3}$&$(2m) \bmod 3$&$\frac{7m - 6n + (2m) \bmod 3}{3}$\\
		\V{11c}&$\frac{ 3n-2m - m \bmod 3}{3}$&$0$&$\frac{7m - 6n - m \bmod 3}{3}$\\
		\V{12a}&$0$&$3n-3m + 1$&$0$\\
		\V{12b}&$0$&$0$&$3m - 3n - 1$\\
		\V{12c}&$1$&$0$&$3m - 3n + 1$\\\hline
	\end{tabular}
\end{table}

\begin{table}[h!]
	\centering
	\caption{Conditions on $n$ and $m$ for the realizability of each V$i$ of Table \ref{vi} and for them to be an extreme point of $\mathcal{P}_{n,m}$ (in addition to condition (\ref{m>12}))}\label{ConditionsVi}~\\[2ex]
	\begin{tabular}{|c|c|c|}\hline		
		Id&Conditions for realizability&Conditions  for being an extreme point of $\mathcal{P}_{n,m}$\\\hline	
		\V{1}&$n \le m \le \lfloor\frac{6n}{5}\rfloor$&$n \le m \le \lfloor\frac{6n}{5}\rfloor$\\
		\V{2}&$m \le \lfloor\frac{6n-8}{5}\rfloor$&$n-1 = m$ or ($m = \lfloor\frac{6n-8}{5}\rfloor$ and $m \bmod 6 \neq 0$)\\
		\V{3}&$n + 1 \le m \le \lfloor\frac{6n+1}{5}\rfloor$&$m = n+1$ or ($m = \lfloor\frac{6n+1}{5}\rfloor$ and $m \bmod 6 = 5$)\\
		\V{6}&$\lfloor\frac{6n+4}{5}\rfloor \le m $&$\lfloor\frac{6n+4}{5}\rfloor \le m$\\
		\V{7a}&$m \le \frac{6n-3((m - 2n) \bmod 4)}{5}$&$m \le \frac{6n-3((m - 2n) \bmod 4)}{5}$\\
		\V{7b}&$m \le \lfloor\frac{6n+3}{5}\rfloor$
		&$m \le \lfloor\frac{6n+3}{5}\rfloor$\\
		\V{7c}&$m \le \lfloor\frac{6n}{5}\rfloor$&$m \le \lfloor\frac{6n}{5}\rfloor$\\
		\V{8c}&no condition&no condition\\
		\V{8d}&$n \bmod 2 = 1$&$n \bmod 2 = 1$\\
		\V{9a}&$n+1 \le m \le \lfloor\frac{21n+13}{20}\rfloor$&$m = n+1$\\
		\V{9b}&$n+1 \le m \le \lfloor\frac{12n+3}{11}\rfloor$&$n{+}1 {=} m$ or $(m {=} \lfloor\frac{12n{+}3}{11}\rfloor$ and $m \bmod 12 \in \{9, 10, 11\})$\\
		\V{9c}&$n+1 \le m \le \lfloor\frac{9n+3}{8}\rfloor$&$n+1 = m$ or ($m = \lfloor\frac{9n+3}{8}\rfloor$ and $m \bmod 9 = 6$)\\
		\V{10a}&$m \le \lfloor\frac{6n}{5}\rfloor$&$m \le \lfloor\frac{6n}{5}\rfloor$\\
		\V{10b}&$m \le \frac{6n-3((m - 2n) \bmod 4)}{5}$&$m \le \frac{6n-3((m - 2n) \bmod 4)}{5}$\\\
		\V{10c}&$m \le \lfloor\frac{6n+2}{5}\rfloor$&$m \le \lfloor\frac{6n+2}{5}\rfloor$\\
		\V{11a}&no condition&no condition\\
		\V{11b}&no condition&no condition\\
		\V{11c}&no condition&no condition\\
		\V{12a}&$m \le n$&$m = n-1$\\
		\V{12b}&$n+2 \le m$&$n+2 \le m$\\
		\V{12c}&$n+2 \le m$&$n+2 \le m$\\\hline
	\end{tabular}
\end{table}

Given $n$, $m$ and a point V$=(a,b,c)$ with three non-negative components, it is shown in \cite{Hansen2017} that a necessary and sufficient condition for V to be realizable for $(n,m)$ is that the following inequalities are valid
\begin{align}
\label{Hansen1}  m_{33} & \leq \tfrac{n_3(n_3-1)}{2} &&\mbox{ if } n_3 = 1, 2 \mbox{ or } 3,\\[-2pt]
\label{Hansen2}  m_{22} & \leq \tfrac{n_2(n_2-1)}{2}  &&\mbox{ if }  n_2 = 1 \mbox{ or }  2,\\[-2pt]
\label{Hansen3}  m_{23} & \leq n_2n_3  &&\mbox{ if }  n_2 = 1 \mbox{ or }  2 \mbox{ and }  n_3 = 1,\\[-2pt]
\label{Hansen4}  m_{23} & \geq \delta(n_2) + \delta(n_3) - 1,\\[-2pt]
\label{Hansen5}  m_{23} + m_{33} & \geq n_3 + \delta(n_2) - 1,\\[-2pt]
\label{Hansen6}  m_{22} + m_{23} & \geq n_2 + \delta(n_3) - 1,\\[-2pt]
\label{Hansen7}  m_{22} + m_{23} + m_{33} & \geq n_2 + n_3 - 1.
\end{align}
where $m_{12}=a$, $m_{13}=b$, $m_{33}=c$, $n_1, n_2, n_3,n,m,m_{22}$ and $m_{23}$ are derived from $m_{12},m_{13},m_{33}$ from equations (\ref{eq_n1})-(\ref{get_x23}) of Section \ref{sec:basic}, and where
$$
\delta(x) = \left\{ 
\begin{array}{ll}
1 & \mbox{if } x \ge 1,\\
0 & \mbox{otherwise.}
\end{array}
\right.
$$
Condition~\eqref{Hansen7} is equivalent to $m - m_{12} - m_{13} \geq n - m_{12} - m_{13} - 1$, which is equivalent to $m \geq n - 1$. Hence, given a point V$i=(m_{12},m_{13},m_{33})$ in Table \ref{vi}, and given $n$ and $m$ that satisfy the conditions for V$i$ in the second column of Table \ref{ConditionsVi}, we can prove that V$i$ is realizable by showing that inequalities (\ref{Hansen1})-(\ref{Hansen6}) are satisfied. Let us illustrate this with \V{1}.

\begin{lem}\label{000}
	Let $n$ and $m$ be integers satisfying $\max\{12,n\}\le m \le \lfloor\frac{6n}{5}\rfloor$. Then, the point
	\emph{\V{1}=(0,0,0)} is realizable for $(n,m)$.
\end{lem}

\begin{proof}
	It follows from equations (\ref{eq_n1})-(\ref{get_x23}) that $m_{22} = 6n-5m$, $m_{23} = 6m - 6n$, $n_1 = 0$, $n_2 = 3n - 2m$ and $n_3 = 2m - 2n$. Hence, all $m_{ij}$ and all $n_i$ are non-negative and integers. Observe also that $m\ge 12$ implies $n\ge 10$ which implies $m\leq\frac{6n}{5}\leq \frac{3n-3}{2}$, and we thus have $n_2=3n-2m\ge 3$.
	\begin{itemize}\setlength\itemsep{0pt}
		\item Inequality (\ref{Hansen1}) is  satisfied since $\frac{n_3(n_3-1)}{2}\geq 0=m_{33}$.
		\item There is no need to check inequalities (\ref{Hansen2}) and (\ref{Hansen3}) since $n_2>2$.
		\item $n_3>0$ if and only if $m>n$. Hence, inequality (\ref{Hansen4}) is satisfied since $m_{23}=6(m-n)\geq \delta(n_3)=\delta(n_2)+\delta(n_3)-1$.
		\item Inequality (\ref{Hansen5}) is satisfied since $m_{23}+m_{33}=6(m-n)\geq 2(m-n)=n_3+\delta(n_2)-1$.
		\item Inequality (\ref{Hansen6}) is satisfied since $m_{22}+m_{23}=m\geq n\geq n_2\geq n_2+\delta(n_3)-1$.\qedhere
	\end{itemize}
\end{proof}

As second illustration, we show that  points \V{8c} and \V{8d} are realizable for $(n,m)$ if the conditions of the second column of Table \ref{ConditionsVi} are satisfied. For this purpose, we consider the point V which is defined as follows:
$$\hspace{-0.3cm}\mbox{V}{=}\left(b,
\frac{3n{-}2m{-}3b{+}(b{-}1)(n \bmod 2)}{2}, \frac{4m{-}3n{-}b{+}(b{-}2a{-}1)(n \bmod 2)}{2}\right)$$
where $a,b$ are binary variables such that $a+b=1$.
Note that V is equal to \V{8c} if $a=1$, and  to \V{8d} if $b=1$.
\begin{lem}
	Let $n$ and $m$ be two integers satisfying condition (\ref{m>12}) and let $a,b$ be binary variables such that $a+b=1$, and  $n\bmod 2=1$ if $b=1$. Then, the point
	\emph{V} is realizable for $(n,m)$.
\end{lem}
\begin{proof}
	It follows from equations (\ref{eq_n1})-(\ref{get_x23}) that
	\begin{align*}
	m_{22} & = 0,\\
	m_{23} & = 2a(n \bmod 2) + b,\\
	n_1 & = \frac{3n-2m-b+(b-1)(n\bmod 2)}{2},\\
	n_2 & = b - (b-1)(n \bmod 2),\\
	n_3 & = \frac{2m-n-b+(b-1)(n \bmod 2)}{2}.
	\end{align*}
	Note that all $m_{ij}$ are non-negative. Indeed, 
	\begin{itemize}\setlength\itemsep{0pt}
		\item $m_{12}$, $m_{22}$ and $m_{23}$ are clearly non-negative;
		\item $\max\{12,n-1\}\leq m\leq \frac{3n-3}{2}$ implies $n\geq 9$ and $m_{33}\geq n-7\geq 2$;
		\item $m_{13}\geq 0$ since      
         $m\leq\lfloor\frac{3n-3}{2}\rfloor$ implies $3n-2m\geq 3\ge 3b+(b-1)n\bmod 2$. 
\end{itemize}
		Also, all $m_{ij}$ and all $n_i$ are integers since $n \bmod2{=}1$ when $b{=}1$. Let's now look at inequalities (\ref{Hansen1})-(\ref{Hansen6}): 
	\begin{itemize}\setlength\itemsep{0pt}
	\item Since $m \geq 12$ and $m\geq n-1$, we have $n_3\ge \frac{2m-n-3}{2}\ge\frac{m-4}{2}\geq 4$, meaning that inequalities (\ref{Hansen1}) and (\ref{Hansen3}) do not need to be checked;
	\item inequality (\ref{Hansen2})
	  is  satisfied. Indeed, $m_{22}=0\leq \frac{n_2(n_2-1)}{2}$ for $n_2\in\{1,2\}$;
	 \item inequality (\ref{Hansen6}) is satisfied since 
	$m_{22} + m_{23} - n_2 + 1 = 1 + a(n \bmod 2) \geq 1 \geq \delta(n_3)$;
	\item to show that inequalities (\ref{Hansen4}) and (\ref{Hansen5}) are satisfied, we consider two cases:
	\begin{itemize}[nosep]
		\item if $a=0$, then
		\begin{itemize}[nosep]
		\item $m_{23} + 1 = b + 1 \geq 2 \geq \delta(n_2) + \delta(n_3)$;
		\item $m_{23} + m_{33} - n_3 + 1 = m - n +b + 1 \geq 1 \geq \delta(n_2)$;
		\end{itemize}
		\item if $a = 1$, then $n_2=\delta(n_2)=n\bmod 2$, which implies
		\begin{itemize}
		\item $m_{23} + 1 = 2(n \bmod 2) + 1 \geq \delta(n_2) + \delta(n_3)$;
		\item $m_{23} + m_{33} - n_3 + 1 = m - n + (n \bmod 2) + 1 \geq \delta(n_2)$.\qedhere
		\end{itemize}
	\end{itemize}
\end{itemize}\end{proof}

The two previous lemmas show that 3 of the 21 points of Table \ref{vi} are realizable. A similar proof can easily be obtained for the 18 other points. A simple way to get a proof that a V$i$ is realizable for $(n,m)$ is to prove that it is impossible to satisfy the conditions of the second column of Table \ref{ConditionsVi} for this V$i$ while violating at least one of the inequalities (\ref{Hansen1})-(\ref{Hansen6}). This can be done using a Satisfiability Modulo Theories (SMT) solver like Z3~\cite{z3}. 

Note that another way of proving that a point $(m_{12},m_{13},m_{33})$ is realizable for $(n,m)$ is to exhibit a chemical graph of order $n$, size $m$, and with $m_{12}$ 12-edges, $m_{13}$ 13-edges and $m_{33}$ 33-edges. For example, we could have proved Lemma \ref{000} by observing that the following graph shows that $(0,0,0)$ is realizable for $(n,m)$ when  $\max\{12,n\}\le m \le \lfloor\frac{6n}{5}\rfloor$: consider a cycle with vertices $v_1,\ldots,v_{2n-m}$ and edges $v_iv_{i+1}$ ($1\le i <2n-m$) and $v_1v_{2n-m}$; for $j=1,\ldots, m-n$, add a vertex $w_j$ linked to $v_{2j-1}$ and to $v_{2n-m-2j+1}$. For illustration, such a graph is shown in Figure \ref{figure5} for $(n,m)=(20,23)$. 

We want to mention here that we provide a webpage ChemicHull \cite{ChemicHull} which, given $n$, $m$ and one of the 21 points V$i=(a,b,c)$, exhibits a chemical graph of order $n$, size $m$, with $a$ 12-edges, $b$ 13-edges, and with $c$ 33-edges.

\begin{figure}
\centering\includegraphics[scale=0.8]{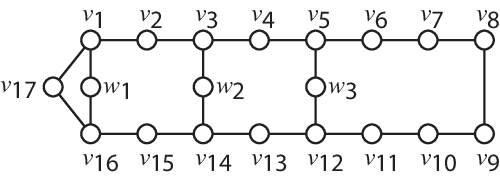}
	\vspace{-0.2cm}\caption{A chemical graph which shows that $(0,0,0)$ is realizable for $(n,m)=(20,23)$}
	\label{figure5}
\end{figure}

\section{The inequalities of Table \ref{facets} are facet defining and induce polytopes}\label{sec:facet}

We first prove that each F$i$ of Table \ref{facets} is facet defining. For this purpose, it is sufficient to give three affinely independent realizable points on the hyperplane defined by F$i$. To prove that three points $p_1,p_2,p_3$ in $\mathbb{R}^3$ are affinely independent, it is sufficient to prove that $p_2-p_1$ is not a multiple of $p_3-p_1$. As we will see, the 21 points of the previous section are sufficient for this purpose.

We give in Table \ref{affine} three points of Table \ref{vi} for every possible case of each F$i$. For example, \G{11} is defined when $m\le \frac{6n-5}{5}$ and $(m\bmod 3)=2$. This is equivalent to say that either $m=\frac{6n-5}{5}$ or $m\le \frac{6n-8}{5}$ and $(m\bmod 3)=2$. We show here how to prove that there are three affinely independent realizable points in each case for \G{11}.
\begin{itemize}\setlength\itemsep{2pt}
    \item if $m=\frac{6n-5}{5}$, then according to Table \ref{ConditionsVi}, \V{7c}=(1,0,0), \V{10a}=(0,1,0) and \V{10c}=(0,2,1) are realizable. Moreover, they belong to the hyperplane defined by \G{11} since $5m-6n+2=-3m_{12}-3m_{13}+3m_{33}=-3$ for these three points. They are affinely independent since \V{10a}-\V{7c}=(-1,1,0) is not a multiple of \V{10c}-\V{7c}=(-1,2,1);
    \item if $m\le \frac{6n-8}{5}$ and $(m\bmod 3)=2$, then according to Table \ref{ConditionsVi}, \V{10a}=$(0,\frac{6n-5m-2}{3},0)$, \V{10b}=$(2,\frac{6n-5m-8}{3},0)$ and \V{10c}=$(0,\frac{6n-5m+1}{3},1)$ are realizable. Moreover, they belong to the hyperplane defined by \G{11} since $-3m_{12}-3m_{13}+3m_{33}=5m-6n+2$ for these three points. They are affinely independent since \V{10b}-\V{10a}=(2,-2,0) is not a multiple of \V{10c}-\V{10a}=(0,1,1).
    \end{itemize}

    It is an easy exercise to show that all triplets of points given in Table \ref{affine} are realizable, belong to the indicated hyperplane and are affinely independent.

\begin{table}
	\centering
	\caption{Three affinely independent points for every inequality F$i$}\label{affine}~\\
\renewcommand{\arraystretch}
{0.9}	\begin{tabular}{|c|lll|c|}\hline
		Id&\multicolumn{3}{c|}{Three points}&Conditions on $n$ and $m$\\\hline
		\multirow{3}*{\G{2}}&\V{6}&\V{8c}&\V{12b}&$\frac{6n+1}{5}\le m$\\
		&\V{1}&\V{8c}&\V{10c}&$n\le m\le\frac{6n}{5}$\\
		&\V{8c}&\V{10a}&\V{12a}&$n-1=m$\\\hline
		\multirow{4}*{\G{3}}&\V{6}&\V{11a}&\V{12b}&$\frac{6n+1}{5}\le m$\\
		&\V{1}&\V{11a}&\V{12b}&$\frac{6n-3}{5}\le m\le\frac{6n}{5}$\\
		&\V{1}&\V{7c}&\V{11a}&$n\le m\le\frac{6n-4}{5}$\\
		&\V{2}&\V{7c}&\V{11c}&$n-1=m$\\\hline
		\multirow{2}*{\G{5}}&\V{1}&\V{7c}&\V{10a}&$n\le m \le \frac{6n-4}{5}$\\
		&\V{2}&\V{7c}&\V{10b}&$n-1=m$\\\hline				
		\G{9}&\V{11a}&\V{11b}&\V{11c}&$(m \bmod 3)=2$\\\hline
		\G{8}&\V{11a}&\V{11b}&\V{11c}&$(m \bmod 3)=1$\\\hline
		\multirow{2}*{\G{1}}&\V{6}&\V{8c}&\V{11a}&$\frac{6n+3}{5}\le m$\\		&\V{8c}&\V{10c}&\V{11a}&$m\le\frac{6n+2}{5}$\\\hline
		\multirow{2}*{\G{11}}&
    \V{7c}&\V{10a}&\V{10c}&$ m=\frac{6n-5}{5},$\\
		&
\V{10a}&\V{10b}&\V{10c}&$(m\le\frac{6n-8}{5})\land ((m \bmod 3)=2)$\\\hline
		\multirow{2}*{\G{14}}&
\V{7b}&\V{10a}&\V{10c}&$m=\frac{6n-2}{5}$\\
&\V{7a}&\V{7b}&\V{7c}&$ (m\le\frac{6n-6}{5})\land((m-2n) \bmod 4=2)$\\\hline		\G{13}&\V{7a}&\V{7b}&\V{7c}&$(m\le \frac{6n-3}{5})\land ((m-2n) \bmod 4=1)$\\\hline
		\multirow{2}*{\G{12}}		&\V{7b}&\V{7c}&\V{10c}&$m\in\{\frac{6n-5}{5},\frac{6n-1}{5}\}$\\
		&\V{7a}&\V{7b}&\V{7c}&$ (m\le\frac{6n-9}{5})\land ((m-2n) \bmod 4=3)$\\
		\hline
		\G{10}&\V{10a}&\V{10b}&\V{10c}&$(m\le \frac{6n-4}{5})\land ((m \bmod 3)=1)$\\\hline
		\multirow{8}*{\G{4}}&\V{11b}&\V{11c}&\V{12c}&$(n+2\le m \le\frac{3n-6}{2})\land ((m \bmod 3) \neq 0)$\\
		&\V{8c}&\V{11a}&\V{12c}&$(n+2\le m \le\frac{3n-6}{2})\land ((m \bmod 3) = 0)\land ((n\bmod 2)=0)$\\
		&\V{8d}&\V{11a}&\V{12c}&$(n+2\le m \le\frac{3n-7}{2})\land ((m \bmod 3) = 0)\land ((n\bmod 2)=1)$\\
		&\V{8c}&\V{9a}&\V{11b}&$(n+1=m)\land ((n\bmod 2)=0)$\\
		&\V{8d}&\V{9a}&\V{11b}&$(n+1=m)\land ((n\bmod 2)=1)$\\
		&\V{1}&\V{11b}&\V{11c}&$(n=m)\land ((n\bmod 3) \neq 0)$\\
		&\V{1}&\V{8c}&\V{11a}&$(n=m)\land ((n\bmod 6) = 0)$\\
		&\V{1}&\V{8d}&\V{11a}&$(n=m)\land ((n\bmod 6) = 3)$\\\hline
		\multirow{3}*{\G{6}}&\V{8c}&\V{8d}&\V{12b}&$(n+2\le m\le\frac{3n-3}{2})\land ((n\bmod 2)=1)$\\
		&\V{8c}&\V{8d}&\V{9c}&$(n+1=m)\land ((n\bmod 2)=1)$\\
		&\V{2}&\V{8c}&\V{8d}&$(n-1=m)\land ((n\bmod 2)=1)$\\\hline
		\G{18}&\V{1}&\V{8c}&\V{8d}&$(n=m)\land ((n\bmod 2) =1)$\\\hline
				\G{7}&\V{8c}&\V{9a}&\V{9c}&$(n+1=m)\land((n \bmod 2)=0)$\\\hline
		\G{7bis}&\V{8c}&\V{12b}&\V{12c}&$(n+2\le m)\land((n \bmod 2)=0)$\\\hline
		\G{17}&\V{2}&\V{8c}&\V{12a}&$(n-1=m)\land ((n\bmod 2) =0)$\\\hline
		\G{19}&\V{2}&\V{11b}&\V{11c}&$(n-1=m)\land ((n\bmod 3) =0)$\\\hline
		\G{20}&\V{2}&\V{11b}&\V{11c}&$(n-1=m)\land ((n\bmod 3) =2)$\\\hline
        		\multirow{2}*{\G{15}}&\V{2}&\V{8d}&\V{11b}&$ (n-1=m)\land((n \bmod 2)=1)$\\
		&\V{2}&\V{8c}&\V{11b}&$(n-1=m)\land((n \bmod 2)=0)$\\\hline
        \G{16}&\V{9a}&\V{9b}&\V{9c}&$n+1=m$\\\hline

	\end{tabular}
\end{table}

In what follows, given $n$ and $m$, we say that a facet F$i$ is \emph{active} for $(n,m)$ if the conditions associated with F$i$ in Table \ref{facets} are satisfied. Let $\mathcal{P}'_{n,m}$ be the polyhedron defined by the active facets F$i$ for $(n,m)$. We now prove that $\mathcal{P}'_{n,m}$ is a polytope since it is bounded for all pairs $(n,m)$ satisfying condition (\ref{m>12}).

\begin{thm}
	$\mathcal{P}'_{n, m}$ is bounded for all $(n,m)$ with $\max\{12,n-1\} \le  m \le \lfloor \frac{3n-3}{2} \rfloor$.
\end{thm}
\begin{proof}
    Note first that \G{2}, \G{3} and \G{1} are active for all pairs $(n,m)$ satisfying condition (\ref{m>12}). Hence, $m_{12}\ge 0$,   $m_{13}\ge 0$, and  $m_{33}\ge 5m-6n$. It remains to prove that $m_{12}$, $m_{13}$ and $m_{33}$ are upper bounded. Consider \G{4}, \G{6}, \G{7bis} and \G{15}. At least one of these four facets is active for $(n,m)$. Indeed:
	\begin{itemize}\setlength\itemsep{0pt}
		\item if $n-1=m$ then \G{15} is active;
		\item if $n\le m\le \frac{3n-6}{2} $, then \G{4} is active;
		\item if $m\ge \frac{3n-5}{2}$ and $n$ is odd then \G{6} is active;
		\item if $m\ge\frac{3n-5}{2}$ and $n$ is even then \G{7bis} is active.
	\end{itemize}
    \newpage We can rewrite each of the above four inequalities as follows:
	\begin{equation}\label{borne}
	p_{12}(n,m)m_{12}+p_{13}(n,m)m_{13}-p_{33}(n,m)m_{33}\geq b(n,m)
	\end{equation}
	where $p_{12}(n,m)$, $p_{13}(n,m)$, $p_{33}(n,m)$ and $b(n,m)$
	are polynomials in $n$ and $m$. We now prove that 
    if the pair $(n,m)$ satisfies condition (\ref{m>12}), then 
    $-4p_{33}(n,m)+p_{12}(n,m)<0$ and $-3p_{33}(n,m)+p_{13}(n,m)<0$:
	\begin{itemize}\setlength\itemsep{1pt}
		\item  for \G{4}, we have $p_{12}(n,m)=p_{13}(n,m)=p_{33}(n,m)=1$, which gives $-4p_{33}(n,m)+p_{12}(n,m)=-3$ and $-3p_{33}(n,m)+p_{13}(n,m)=-2$;
		\item  for \G{6} we have $p_{12}(n,m)=2$, $p_{13}(n,m)=p_{33}(n,m)=1$, which gives $-4p_{33}(n,m)+p_{12}(n,m)=-3p_{33}(n,m)+p_{13}(n,m)=-2$;
		\item  for \G{7bis}, it follows from (\ref{7bis}) that $p_{12}(n,m)=2(3n-2m)$, $p_{13}(n,m)=3n-2m+2$ and $p_{33}(n,m)=3n-2m$, which gives $-4p_{33}(n,m)+p_{12}(n,m)=2(2m-3n)\leq -6$ and $-3p_{33}(n,m)+p_{13}(n,m)=2(2m-3n+1)\leq -4$;
		\item  for \G{15} we have $p_{12}(n,m)=n-7$, $p_{13}(n,m)=n-6$ and $p_{33}(n,m)=n-4$, which gives $-4p_{33}(n,m)+p_{12}(n,m)=-3(n-3)<0$ and $-3p_{33}(n,m)+p_{13}(n,m)=-2(n-3)<0$.
	\end{itemize}
	Hence, by adding $p_{33}(n,m)$ times inequality \G{1} to (\ref{borne}) we get
	$$
	\big(-4p_{33}(n,m)+p_{12}(n,m)\big)m_{12}+\big(-3p_{33}(n,m)+p_{13}(n,m)\big)m_{13}\geq p_{33}(n,m)(5m-6n)+b(n,m),$$
	which implies that $m_{12}$ and $m_{13}$ are upper bounded by a polynomial that depends on $n$ and $m$ (since \G{2} and \G{3} forbid negative values for these two variables).
	But (\ref{borne}) is equivalent to
	$$p_{33}(n,m)m_{33}\leq p_{12}(n,m)m_{12}+p_{13}(n,m)m_{13}-b(n,m),$$
	which implies that $m_{33}$ is also upper bounded by a polyomial that depends on $n$ and $m$.
\end{proof}

\section{Extreme points of $\mathcal{P}_{n,m}$}\label{sec:extreme}
As proved in Sections \ref{sec:valid} and \ref{sec:facet}, the 21 inequalities of Table \ref{facets} are valid and facet defining. They are defined with constraints on $n$ and $m$. Hence, depending on $n$ and $m$, not all of facets are active for $(n,m)$. The set of integer pairs $(n,m)$ that satisfy condition (\ref{m>12}) can be partitioned into 75 subsets, each one having a different set of active facets for $(n,m)$. Hence, there are 75 polytope types, each one having its subset of  active facets. These polytope types are listed in Table \ref{polytopes}. We indicate in the first column the Id of each polytope type, we then give the set of active facets, and we finally indicate the constraints on $n$ and $m$. 

Note that we do not indicate redundant constraints on $n$ and $m$. For example, if $\frac{6n-3}{5}\le m<\frac{6n}{5}$ and $m\bmod 3=0$, then $m=\frac{6n-3}{5}$ and, whatever the parity of $n$, we necessarily have $1=(-3)\bmod 4= (5m-6n)\bmod 4=(m-2n)\bmod 4$. This explains the existence of facet \G{13} for polytope types P61 and P64, even if it is not explicitly mentioned that $(m-2n)\bmod 4=1$.
\newline

{\small\renewcommand*{\arraystretch}{0.9}\begin{longtable}{|l|l|ccc|}
\caption{Polytopes  with their facets and their conditions on $n$ and $m$ (in addition to condition (\ref{m>12}))}\label{polytopes}
\endfirsthead

\multicolumn{5}{c}{(Table \ref{polytopes} continued)} \\

\endhead

\endfoot

\endlastfoot

\multicolumn{5}{c}{$m=n -1$}\\\cline{1-3}
\multicolumn{1}{|c}{Id}&\multicolumn{1}{c}{facets}&\multicolumn{1}{c|}{$n\bmod 12$}&\multicolumn{2}{c}{} \\ \cline{1-3}
\multicolumn{1}{|c|}{P1}&\G{2}, \G{3}, \G{5}, \G{9}, \G{1}, \G{11}, \G{12}, \G{17}, \G{19}, \G{15} &\multicolumn{1}{c|}{0}&\multicolumn{2}{c}{} \\
\multicolumn{1}{|c|}{P2}&\G{2}, \G{3}, \G{5}, \G{1}, \G{14}, \G{6}, \G{15} &\multicolumn{1}{c|}{1}&\multicolumn{2}{c}{} \\
\multicolumn{1}{|c|}{P3}&\G{2}, \G{3}, \G{5}, \G{8}, \G{1}, \G{13}, \G{10}, \G{17}, \G{20}, \G{15} &\multicolumn{1}{c|}{2}&\multicolumn{2}{c}{} \\
\multicolumn{1}{|c|}{P4}&\G{2}, \G{3}, \G{5}, \G{9}, \G{1}, \G{11}, \G{6}, \G{19}, \G{15} &\multicolumn{1}{c|}{3}&\multicolumn{2}{c}{} \\
\multicolumn{1}{|c|}{P5}&\G{2}, \G{3}, \G{5}, \G{1}, \G{12}, \G{17}, \G{15} &\multicolumn{1}{c|}{4}&\multicolumn{2}{c}{} \\
\multicolumn{1}{|c|}{P6}&\G{2}, \G{3}, \G{5}, \G{8}, \G{1}, \G{14}, \G{10}, \G{6}, \G{20}, \G{15} &\multicolumn{1}{c|}{5}&\multicolumn{2}{c}{} \\
\multicolumn{1}{|c|}{P7}&\G{2}, \G{3}, \G{5}, \G{9}, \G{1}, \G{11}, \G{13}, \G{17}, \G{19}, \G{15} &\multicolumn{1}{c|}{6}&\multicolumn{2}{c}{} \\
\multicolumn{1}{|c|}{P8}&\G{2}, \G{3}, \G{5}, \G{1}, \G{6}, \G{15} &\multicolumn{1}{c|}{7}&\multicolumn{2}{c}{} \\
\multicolumn{1}{|c|}{P9}&\G{2}, \G{3}, \G{5}, \G{8}, \G{1}, \G{12}, \G{10}, \G{17}, \G{20}, \G{15} &\multicolumn{1}{c|}{8}&\multicolumn{2}{c}{} \\
\multicolumn{1}{|c|}{P10}&\G{2}, \G{3}, \G{5}, \G{9}, \G{1}, \G{11}, \G{14}, \G{6}, \G{19}, \G{15} &\multicolumn{1}{c|}{9}&\multicolumn{2}{c}{} \\
\multicolumn{1}{|c|}{P11}&\G{2}, \G{3}, \G{5}, \G{1}, \G{13}, \G{17}, \G{15} &\multicolumn{1}{c|}{10}&\multicolumn{2}{c}{}\\
\multicolumn{1}{|c|}{P12}&\G{2}, \G{3}, \G{5}, \G{8}, \G{1}, \G{10}, \G{6}, \G{20}, \G{15} &\multicolumn{1}{c|}{11}&\multicolumn{2}{c}{} \\
\cline{1-3}
\multicolumn{5}{c}{\vspace{-2pt}}\\
\multicolumn{5}{c}{$m=n$}\\\cline{1-3}
\multicolumn{1}{|c}{Id}&\multicolumn{1}{c}{facets}&\multicolumn{1}{c|}{$n\bmod 12$}&\multicolumn{2}{c}{}\\ \cline{1-3}
\multicolumn{1}{|c|}{P13}&\G{2}, \G{3}, \G{5}, \G{1}, \G{4} &\multicolumn{1}{c|}{0}&\multicolumn{2}{c}{}  \\
\multicolumn{1}{|c|}{P14}&\G{2}, \G{3}, \G{5}, \G{8}, \G{1}, \G{12}, \G{10}, \G{4}, \G{18} &\multicolumn{1}{c|}{1}&\multicolumn{2}{c}{} \\
\multicolumn{1}{|c|}{P15}&\G{2}, \G{3}, \G{5}, \G{9}, \G{1}, \G{11}, \G{14}, \G{4} &\multicolumn{1}{c|}{2}&\multicolumn{2}{c}{}\\
\multicolumn{1}{|c|}{P16}&\G{2}, \G{3}, \G{5}, \G{1}, \G{13}, \G{4}, \G{18}&\multicolumn{1}{c|}{3}&\multicolumn{2}{c}{} \\
\multicolumn{1}{|c|}{P17}&\G{2}, \G{3}, \G{5}, \G{8}, \G{1}, \G{10}, \G{4}&\multicolumn{1}{c|}{4}&\multicolumn{2}{c}{} \\
\multicolumn{1}{|c|}{P18}&\G{2}, \G{3}, \G{5}, \G{9}, \G{1}, \G{11}, \G{12}, \G{4}, \G{18}&\multicolumn{1}{c|}{5}&\multicolumn{2}{c}{} \\
\multicolumn{1}{|c|}{P19}&\G{2}, \G{3}, \G{5}, \G{1}, \G{14}, \G{4}&\multicolumn{1}{c|}{6}&\multicolumn{2}{c}{}  \\
\multicolumn{1}{|c|}{P20}&\G{2}, \G{3}, \G{5}, \G{8}, \G{1}, \G{13}, \G{10}, \G{4}, \G{18}&\multicolumn{1}{c|}{7}&\multicolumn{2}{c}{} \\
\multicolumn{1}{|c|}{P21}&\G{2}, \G{3}, \G{5}, \G{9}, \G{1}, \G{11}, \G{4}&\multicolumn{1}{c|}{8}&\multicolumn{2}{c}{} \\
\multicolumn{1}{|c|}{P22}&\G{2}, \G{3}, \G{5}, \G{1}, \G{12}, \G{4}, \G{18}&\multicolumn{1}{c|}{9}&\multicolumn{2}{c}{} \\
\multicolumn{1}{|c|}{P23}&\G{2}, \G{3}, \G{5}, \G{8}, \G{1}, \G{14}, \G{10}, \G{4}&\multicolumn{1}{c|}{10}&\multicolumn{2}{c}{} \\
\multicolumn{1}{|c|}{P24}&\G{2}, \G{3}, \G{5}, \G{9}, \G{1}, \G{11}, \G{13}, \G{4}, \G{18}&\multicolumn{1}{c|}{11}&\multicolumn{2}{c}{} \\
\cline{1-3}
\multicolumn{5}{c}{\vspace{-2pt}}\\
\multicolumn{5}{c}
 {$m=n+1$}\\\cline{1-3}
\multicolumn{1}{|c}{Id}&\multicolumn{1}{c}{facets}&\multicolumn{1}{c|}{$n\bmod 12$}&\multicolumn{2}{c}{}\\ \cline{1-3}
\multicolumn{1}{|c|}{P25}&\G{2}, \G{3}, \G{5}, \G{8}, \G{1}, \G{13}, \G{10}, \G{4}, \G{7}, \G{16} &\multicolumn{1}{c|}{0}&\multicolumn{2}{c}{}\\
\multicolumn{1}{|c|}{P26}&\G{2}, \G{3}, \G{5}, \G{9}, \G{1}, \G{11}, \G{4}, \G{6}, \G{16} &\multicolumn{1}{c|}{1}&\multicolumn{2}{c}{} \\
\multicolumn{1}{|c|}{P27}&\G{2}, \G{3}, \G{5}, \G{1}, \G{12}, \G{4}, \G{7}, \G{16} &\multicolumn{1}{c|}{2}&\multicolumn{2}{c}{} \\
\multicolumn{1}{|c|}{P28}&\G{2}, \G{3}, \G{5}, \G{8}, \G{1}, \G{14}, \G{10}, \G{4}, \G{6}, \G{16} &\multicolumn{1}{c|}{3}&\multicolumn{2}{c}{} \\
\multicolumn{1}{|c|}{P29}&\G{2}, \G{3}, \G{5}, \G{9}, \G{1}, \G{11}, \G{13}, \G{4}, \G{7}, \G{16} &\multicolumn{1}{c|}{4}&\multicolumn{2}{c}{} \\
\multicolumn{1}{|c|}{P30}&\G{2}, \G{3}, \G{5}, \G{1}, \G{4}, \G{6}, \G{16} &\multicolumn{1}{c|}{5}&\multicolumn{2}{c}{} \\
\multicolumn{1}{|c|}{P31}&\G{2}, \G{3}, \G{5}, \G{8}, \G{1}, \G{12}, \G{10}, \G{4}, \G{7}, \G{16} &\multicolumn{1}{c|}{6}&\multicolumn{2}{c}{} \\
\multicolumn{1}{|c|}{P32}&\G{2}, \G{3}, \G{5}, \G{9}, \G{1}, \G{11}, \G{14}, \G{4}, \G{6}, \G{16} &\multicolumn{1}{c|}{7}&\multicolumn{2}{c}{} \\
\multicolumn{1}{|c|}{P33}&\G{2}, \G{3}, \G{5}, \G{1}, \G{13}, \G{4}, \G{7}, \G{16} &\multicolumn{1}{c|}{8}&\multicolumn{2}{c}{} \\
\multicolumn{1}{|c|}{P34}&\G{2}, \G{3}, \G{5}, \G{8}, \G{1}, \G{10}, \G{4}, \G{6}, \G{16} &\multicolumn{1}{c|}{9}&\multicolumn{2}{c}{} \\
\multicolumn{1}{|c|}{P35}&\G{2}, \G{3}, \G{5}, \G{9}, \G{1}, \G{11}, \G{12}, \G{4}, \G{7}, \G{16} &\multicolumn{1}{c|}{10}&\multicolumn{2}{c}{} \\
\multicolumn{1}{|c|}{P36}&\G{2}, \G{3}, \G{5}, \G{1}, \G{14}, \G{4}, \G{6}, \G{16} &\multicolumn{1}{c|}{11}&\multicolumn{2}{c}{} \\
\cline{1-3}
\newpage\multicolumn{5}{c}{\vspace{10pt}}\\
\multicolumn{5}{c}{$n + 1<m<\frac{6n-3}{5}$}\\\hline
\multicolumn{1}{|c}{Id}&\multicolumn{1}{c}{facets}&$n\bmod 2$&$m\bmod 3$&$(m{-}2n){\bmod} 4$ \\ \hline
\multicolumn{1}{|c|}{P37}&\G{2}, \G{3}, \G{5}, \G{1}, \G{4}, \G{7bis}&0&0&0  \\
\multicolumn{1}{|c|}{P38}&\G{2}, \G{3}, \G{5}, \G{1}, \G{13}, \G{4}, \G{7bis}&0&0&1 \\
\multicolumn{1}{|c|}{P39}&\G{2}, \G{3}, \G{5}, \G{1}, \G{14}, \G{4}, \G{7bis} &0&0&2 \\
\multicolumn{1}{|c|}{P40}&\G{2}, \G{3}, \G{5}, \G{1}, \G{12}, \G{4}, \G{7bis}&0&0&3  \\
\multicolumn{1}{|c|}{P41}&\G{2}, \G{3}, \G{5}, \G{8}, \G{1}, \G{10}, \G{4}, \G{7bis} &0&1&0 \\
\multicolumn{1}{|c|}{P42}&\G{2}, \G{3}, \G{5}, \G{8}, \G{1}, \G{13}, \G{10}, \G{4}, \G{7bis} &0&1&1 \\
\multicolumn{1}{|c|}{P43}&\G{2}, \G{3}, \G{5}, \G{8}, \G{1}, \G{14}, \G{10}, \G{4}, \G{7bis} &0&1&2 \\
\multicolumn{1}{|c|}{P44}&\G{2}, \G{3}, \G{5}, \G{8}, \G{1}, \G{12}, \G{10}, \G{4}, \G{7bis}&0&1&3\\
\multicolumn{1}{|c|}{P45}&\G{2}, \G{3}, \G{5}, \G{9}, \G{1}, \G{11}, \G{4}, \G{7bis}&0&2&0  \\
\multicolumn{1}{|c|}{P46}&\G{2}, \G{3}, \G{5}, \G{9}, \G{1}, \G{11}, \G{13}, \G{4}, \G{7bis}&0&2&1  \\
\multicolumn{1}{|c|}{P47}&\G{2}, \G{3}, \G{5}, \G{9}, \G{1}, \G{11}, \G{14}, \G{4}, \G{7bis}&0&2&2  \\
\multicolumn{1}{|c|}{P48}&\G{2}, \G{3}, \G{5}, \G{9}, \G{1}, \G{11}, \G{12}, \G{4}, \G{7bis}&0&2&3  \\
\multicolumn{1}{|c|}{P49}&\G{2}, \G{3}, \G{5}, \G{1}, \G{4}, \G{6}&1&0&0  \\
\multicolumn{1}{|c|}{P50}&\G{2}, \G{3}, \G{5}, \G{1}, \G{13}, \G{4}, \G{6} &1&0&1\\
\multicolumn{1}{|c|}{P51}&\G{2}, \G{3}, \G{5}, \G{1}, \G{14}, \G{4}, \G{6}&1&0&2  \\
\multicolumn{1}{|c|}{P52}&\G{2}, \G{3}, \G{5}, \G{1}, \G{12}, \G{4}, \G{6} &1&0&3 \\
\multicolumn{1}{|c|}{P53}&\G{2}, \G{3}, \G{5}, \G{8}, \G{1}, \G{10}, \G{4}, \G{6}&1&1&0  \\
\multicolumn{1}{|c|}{P54}&\G{2}, \G{3}, \G{5}, \G{8}, \G{1}, \G{13}, \G{10}, \G{4}, \G{6} &1&1&1 \\
\multicolumn{1}{|c|}{P55}&\G{2}, \G{3}, \G{5}, \G{8}, \G{1}, \G{14}, \G{10}, \G{4}, \G{6}&1&1&2\\
\multicolumn{1}{|c|}{P56}&\G{2}, \G{3}, \G{5}, \G{8}, \G{1}, \G{12}, \G{10}, \G{4}, \G{6} &1&1&3 \\   
\multicolumn{1}{|c|}{P57}&\G{2}, \G{3}, \G{5}, \G{9}, \G{1}, \G{11}, \G{4}, \G{6} &1&2&0 \\
\multicolumn{1}{|c|}{P58}&\G{2}, \G{3}, \G{5}, \G{9}, \G{1}, \G{11}, \G{13}, \G{4}, \G{6}&1&2&1  \\
\multicolumn{1}{|c|}{P59}&\G{2}, \G{3}, \G{5}, \G{9}, \G{1}, \G{11}, \G{14}, \G{4}, \G{6} &1&2&2 \\
\multicolumn{1}{|c|}{P60}&\G{2}, \G{3}, \G{5}, \G{9}, \G{1}, \G{11}, \G{12}, \G{4}, \G{6}&1&2&3  \\
      \hline
\multicolumn{5}{c}{}\\
\multicolumn{5}{c}{$\frac{6n-3}{5} \le m < \frac{6n}{5}$}\\\cline{1-4}
\multicolumn{1}{|c}{Id}&\multicolumn{1}{c}{facets}&$n\bmod 2$&\multicolumn{1}{c|}{$m\bmod 3$}&\multicolumn{1}{c}{} \\ \cline{1-4}
\multicolumn{1}{|c|}{P61}&\G{2}, \G{3}, \G{1}, \G{13}, \G{4}, \G{7bis} &0&\multicolumn{1}{c|}{0}&\multicolumn{1}{c}{}\\
\multicolumn{1}{|c|}{P62}&\G{2}, \G{3}, \G{8}, \G{1}, \G{12}, \G{4}, \G{7bis}&0&\multicolumn{1}{c|}{1}&\multicolumn{1}{c}{}  \\
\multicolumn{1}{|c|}{P63}&\G{2}, \G{3}, \G{9}, \G{1}, \G{14}, \G{4}, \G{7bis}&0&\multicolumn{1}{c|}{2}&\multicolumn{1}{c}{}  \\
\multicolumn{1}{|c|}{P64}&\G{2}, \G{3}, \G{1}, \G{13}, \G{4}, \G{6}&1&\multicolumn{1}{c|}{0}&\multicolumn{1}{c}{}  \\
\multicolumn{1}{|c|}{P65}&\G{2}, \G{3}, \G{8}, \G{1}, \G{12}, \G{4}, \G{6} &1&\multicolumn{1}{c|}{1}&\multicolumn{1}{c}{} \\
\multicolumn{1}{|c|}{P66}&\G{2}, \G{3}, \G{9}, \G{1}, \G{14}, \G{4}, \G{6} &1&\multicolumn{1}{c|}{2}&\multicolumn{1}{c}{} \\
\cline{1-4}
\multicolumn{5}{c}{}\\
\multicolumn{5}{c}{$\frac{6n}{5} \le m < \lfloor\frac{3n}{2}\rfloor-2$}\\\cline{1-4}
\multicolumn{1}{|c}{Id}&\multicolumn{1}{c}{facets}&$n\bmod 2$&\multicolumn{1}{c|}{$m\bmod 3$}&\multicolumn{1}{c}{} \\ \cline{1-4}
\multicolumn{1}{|c|}{P67}&\G{2}, \G{3}, \G{1}, \G{4}, \G{7bis}&0&\multicolumn{1}{c|}{0}&\multicolumn{1}{c}{}  \\
\multicolumn{1}{|c|}{P68}&\G{2}, \G{3}, \G{8}, \G{1}, \G{4}, \G{7bis} &0&\multicolumn{1}{c|}{1}&\multicolumn{1}{c}{} \\
\multicolumn{1}{|c|}{P69}&\G{2}, \G{3}, \G{9}, \G{1}, \G{4}, \G{7bis}&0&\multicolumn{1}{c|}{2}&\multicolumn{1}{c}{}  \\
\multicolumn{1}{|c|}{P70}&\G{2}, \G{3}, \G{1}, \G{4}, \G{6}&1&\multicolumn{1}{c|}{0}&\multicolumn{1}{c}{}  \\
\multicolumn{1}{|c|}{P71}&\G{2}, \G{3}, \G{8}, \G{1}, \G{4}, \G{6}&1&\multicolumn{1}{c|}{1}&\multicolumn{1}{c}{}  \\
\multicolumn{1}{|c|}{P72}&\G{2}, \G{3}, \G{9}, \G{1}, \G{4}, \G{6}&1&\multicolumn{1}{c|}{2}&\multicolumn{1}{c}{}  \\
\cline{1-4}
\multicolumn{5}{c}{}\\
\multicolumn{5}{c}{$\lfloor\frac{3n}{2}\rfloor-2 \le m$}\\\cline{1-3}
\multicolumn{1}{|c}{Id}&\multicolumn{1}{c}{facets}&\multicolumn{1}{c|}{$m\bmod 3$}&\multicolumn{2}{c}{} \\\cline{1-3}
\multicolumn{1}{|c|}{P73}&\G{2}, \G{3}, \G{1}, \G{6} &\multicolumn{1}{c|}{0}&\multicolumn{2}{c}{}\\
\multicolumn{1}{|c|}{P74}&\G{2}, \G{3}, \G{8}, \G{1}, \G{7bis} &\multicolumn{1}{c|}{1}&\multicolumn{2}{c}{}\\
\multicolumn{1}{|c|}{P75}&\G{2}, \G{3}, \G{9}, \G{1}, \G{6} &\multicolumn{1}{c|}{2}&\multicolumn{2}{c}{}\\
\cline{1-3}

\end{longtable}}

As a reminder, given two integers $n$ and $m$ that satisfy condition (\ref{m>12})$, \mathcal{P}'_{n,m}$ is the polytope defined by all active facets for $(n,m)$. We now prove that the set of extreme points of $\mathcal{P}'_{n,m}$ is a subset of the 21 points described in Section \ref{sec:realizable}. We therefore have $\mathcal{P}'_{n,m}=\mathcal{P}_{n,m}$.

It is proved in Section \ref{sec:facet} that $\mathcal{P}'_{n,m}$ is bounded for all $n$ and $m$ that satisfy condition (\ref{m>12}). So, let $I$ be the set of indices $i$ such that F$i$ is active for $(n,m)$. It is sufficient to prove that given any subset of three distinct indices $i_1,i_2,i_3$ in $I$, the point at the intersection of the hyperplanes defined by F$i_1$, F$i_2$ and F$i_3$ (if it exists) is either outside $\mathcal{P}'_{n,m}$, or it is one of the 21 points of Table \ref{vi} which is realizable for $(n,m)$ according to Table \ref{ConditionsVi}.
Let's illustrate this process with an example. 

\begin{lem}\label{P70}
    If $n\bmod 2=1, m\bmod 3=0$ and $\frac{6n}{5}\le m\le \frac{3n-5}{2}$, then  the extreme points of $\mathcal{P}'_{n,m}$ are \emph{\V{6}, \V{8c}, \V{8d}, \V{11a}=\V{11b}=\V{11c}, \V{12b}} and \emph{\V{12c}}. 
\end{lem}
\begin{proof}
Note first that Table \ref{ConditionsVi} shows that \V{6}, \V{8c}, \V{8d}, $\V{11a}=\V{11b}=\V{11c}, \V{12b}$ and \V{12c} are realizable for $(n,m)$. As shown in Table \ref{polytopes}, the conditions on $n$ and $m$ are associated with P70 which has \G{2}, \G{3}, \G{1}, \G{4} and \G{6} as active facets. Hence, $I=\{1,2,6,12,13\}$. Let's consider all triplets ($i_1,i_2,i_3$) of indices in $I$.
\begin{itemize}\setlength\itemsep{-1pt}
	\item $(i_1,i_2,i_3)=(1,2,6)$. The intersection point of the hyperplanes defined by \G{2}, \G{3}, \G{1}, is $\V{6}=(0,0,5m-6n)$. Inequalities \G{4} and \G{6} are satisfied since $m_{12}+m_{13}-m_{33}=2m_{12}+m_{13}-m_{33}=6n-5m>3n-3m+1$. Hence, \V{6} is an extreme point.
	\item $(i_1,i_2,i_3)=(1,2,12)$. The intersection point of the hyperplanes defined by  \G{2}, \G{3}, \G{4} is $(0,0,3m-3n)$. Inequality \G{6} is violated since $2m_{12}+m_{13}-m_{33}=3n-3m<3n-3m+1$. Hence, this point is outside the polytope.
	\item $(i_1,i_2,i_3)=(1,2,13)$. The intersection point of the hyperplanes defined by  \G{2}, \G{3}, \G{6} is $\V{12b}=(0,0,3m-3n-1)$. Inequalities \G{1} and \G{4} are satisfied since $-4m_{12}-3m_{13}+m_{33}=3m-3n-1>5m-6n$ and $m_{12}+m_{13}-m_{33}=3n-3m+1>3n-3m$. Hence, \V{12b} is an extreme point.
\item $(i_1,i_2,i_3)=(1,6,12)$.The intersection point of the hyperplanes defined by  \G{2}, \G{1}, \G{4} is $(0,\frac{3n-2m}{2},\frac{4m-3n}{2})$. Inequality \G{6} is violated since $2m_{12}+m_{13}-m_{33}=3n-3m<3n-3m+1$. Hence, this point is outside the polytope.
	\item $(i_1,i_2,i_3)=(1,6,13)$. The intersection point of the hyperplanes defined by  \G{2}, \G{1}, \G{6} is $\V{8c}=(0,\frac{3n-2m-1}{2},\frac{4m-3n-3}{2})$. Inequalities \G{3} and \G{4} are satisfied since $3n-2m-1>0$ and $m_{12}+m_{13}-m_{33}=3n-3m+1>3n-3m$. Hence, \V{8c} is an extreme point.
	\item $(i_1,i_2,i_3)=(1,12,13)$. There is no intersection point for the hyperplanes defined by \G{2}, \G{4}, \G{6} since \G{2} implies $m_{12}=0$, while \G{6}-\G{4} implies $m_{12}=1$.
	\item $(i_1,i_2,i_3)=(2,6,12)$. The intersection point of the hyperplanes defined by  \G{3}, \G{1}, \G{4} is $\V{11a}=\V{11b}=\V{11c}=(\frac{3n-2m}{3},0,\frac{7m-6n}{3})$. Inequalities \G{2} and \G{6} are satisfied since $3n-2m>0$ and $2m_{12}+m_{13}-m_{33}=3n-3m+\frac{3n-2m}{3}>3n-3m+1$. Hence, \V{11a}=\V{11b}=\V{11c} is an extreme point.
	\item $(i_1,i_2,i_3)=(2,6,13)$. The intersection point of the hyperplanes defined by  \G{3}, \G{1}, \G{6} is $(\frac{3n-2m-1}{2},0,m-2)$. Inequality \G{4} is violated since $m_{12}+m_{13}-m_{33}=3n-3m+\frac{2m-3n+3}{2}\le 3n-3m-1$. Hence, this point is outside the polytope.
	\item $(i_1,i_2,i_3)=(2,12, 13)$. The intersection point of inequalities \G{3}, \G{4}, \G{6} is $\V{12c}=(1,0,3m-3n+1)$. Inequalities \G{2} and \G{1} are satisfied since $1>0$ and $-4m_{12}-3m_{13}+m_{33}=3m-3n-3>5m-6n$. Hence, \V{12c} is an extreme point.
	\item $(i_1,i_2,i_3)=(6,12,13)$. The intersection point of the hyperplanes defined by  \G{1}, \G{4}, \G{6} is $\V{8d}=(1,\frac{3n-2m-3}{2},\frac{4m-3n-1}{2})$. Inequalities \G{2} and \G{3} are satisfied since $1>0$ and $3n-2m-3>0$. Hence, $\V{8d}$ is an extreme point.\qedhere
\end{itemize}
\end{proof}

\begin{table}
\scriptsize	\centering
	\caption{Extreme points for the 75 polytopes of Table \ref{polytopes} when $(n,m)$ satisfies condition (\ref{m>12})}\label{extreme}~\\
\renewcommand{\arraystretch}{0.83}	
\begin{tabular}{|c|l|}
\hline
\vspace{-4pt}&\\
Polytopes & \multicolumn{1}{c|}{Extreme points}\\
\vspace{-4pt}&\\
\hline
\vspace{-4pt}&\\
P1 & \V{2}, \V{7a}, \V{7b}, \V{7c}, \V{8c}, \V{10a}, \V{10b}, \V{10c}, \V{11a}, \V{11b}, \V{11c}, \V{12a} \\
P2 & \V{2}, \V{7a}, \V{7b}, \V{7c}, \V{8c}, \V{8d}, \V{10a}=\V{10b}=\V{10c}, \V{11a}=\V{11b}=\V{11c}, \V{12a} \\
P3 & \V{2}, \V{7a}, \V{7b}, \V{7c}, \V{8c}, \V{10a}, \V{10b}, \V{10c}, \V{11a}, \V{11b}, \V{11c}, \V{12a} \\
P4 & \V{2}, \V{7a}=\V{7b}=\V{7c}, \V{8c}, \V{8d}, \V{10a}, \V{10b}, \V{10c}, \V{11a}, \V{11b}, \V{11c}, \V{12a} \\
P5 & \V{2}, \V{7a}, \V{7b}, \V{7c}, \V{8c}, \V{10a}=\V{10b}=\V{10c}, \V{11a}=\V{11b}=\V{11c}, \V{12a} \\
P6 & \V{2}, \V{7a}, \V{7b}, \V{7c}, \V{8c}, \V{8d}, \V{10a}, \V{10b}, \V{10c}, \V{11a}, \V{11b}, \V{11c}, \V{12a} \\
P7 & \V{2}, \V{7a}, \V{7b}, \V{7c}, \V{8c}, \V{10a}, \V{10b}, \V{10c}, \V{11a}, \V{11b}, \V{11c}, \V{12a} \\
P8 & \V{2}, \V{7a}=\V{7b}=\V{7c}, \V{8c}, \V{8d}, \V{10a}=\V{10b}=\V{10c}, \V{11a}=\V{11b}=\V{11c}, \V{12a} \\
P9 & \V{2}, \V{7a}, \V{7b}, \V{7c}, \V{8c}, \V{10a}, \V{10b}, \V{10c}, \V{11a}, \V{11b}, \V{11c}, \V{12a} \\
P10 & \V{2}, \V{7a}, \V{7b}, \V{7c}, \V{8c}, \V{8d}, \V{10a}, \V{10b}, \V{10c}, \V{11a}, \V{11b}, \V{11c}, \V{12a} \\
P11 & \V{2}, \V{7a}, \V{7b}, \V{7c}, \V{8c}, \V{10a}=\V{10b}=\V{10c}, \V{11a}=\V{11b}=\V{11c}, \V{12a} \\
P12 & \V{2}, \V{7a}=\V{7b}=\V{7c}, \V{8c}, \V{8d}, \V{10a}, \V{10b}, \V{10c}, \V{11a}, \V{11b}, \V{11c}, \V{12a} \\
P13 & \V{1}, \V{7a}=\V{7b}=\V{7c}, \V{8c}, \V{10a}=\V{10b}=\V{10c}, \V{11a}=\V{11b}=\V{11c} \\
P14 & \V{1}, \V{7a}, \V{7b}, \V{7c}, \V{8c}, \V{8d}, \V{10a}, \V{10b}, \V{10c}, \V{11a}, \V{11b}, \V{11c} \\
P15 & \V{1}, \V{7a}, \V{7b}, \V{7c}, \V{8c}, \V{10a}, \V{10b}, \V{10c}, \V{11a}, \V{11b}, \V{11c} \\
P16 & \V{1}, \V{7a}, \V{7b}, \V{7c}, \V{8c}, \V{8d}, \V{10a}=\V{10b}=\V{10c}, \V{11a}=\V{11b}=\V{11c} \\
P17 & \V{1}, \V{7a}=\V{7b}=\V{7c}, \V{8c}, \V{10a}, \V{10b}, \V{10c}, \V{11a}, \V{11b}, \V{11c} \\
P18 & \V{1}, \V{7a}, \V{7b}, \V{7c}, \V{8c}, \V{8d}, \V{10a}, \V{10b}, \V{10c}, \V{11a}, \V{11b}, \V{11c} \\
P19 & \V{1}, \V{7a}, \V{7b}, \V{7c}, \V{8c}, \V{10a}=\V{10b}=\V{10c}, \V{11a}=\V{11b}=\V{11c} \\
P20 & \V{1}, \V{7a}, \V{7b}, \V{7c}, \V{8c}, \V{8d}, \V{10a}, \V{10b}, \V{10c}, \V{11a}, \V{11b}, \V{11c} \\
P21 & \V{1}, \V{7a}=\V{7b}=\V{7c}, \V{8c}, \V{10a}, \V{10b}, \V{10c}, \V{11a}, \V{11b}, \V{11c} \\
P22 & \V{1}, \V{7a}, \V{7b}, \V{7c}, \V{8c}, \V{8d}, \V{10a}=\V{10b}=\V{10c}, \V{11a}=\V{11b}=\V{11c} \\
P23 & \V{1}, \V{7a}, \V{7b}, \V{7c}, \V{8c}, \V{10a}, \V{10b}, \V{10c}, \V{11a}, \V{11b}, \V{11c} \\
P24 & \V{1}, \V{7a}, \V{7b}, \V{7c}, \V{8c}, \V{8d}, \V{10a}, \V{10b}, \V{10c}, \V{11a}, \V{11b}, \V{11c} \\
P25 & \V{1}, \V{3}, \V{7a}, \V{7b}, \V{7c}, \V{8c}, \V{9a}, \V{9b}, \V{9c}, \V{10a}, \V{10b}, \V{10c}, \V{11a}, \V{11b}, \V{11c} \\
P26 & \V{1}, \V{3}, \V{7a}=\V{7b}=\V{7c}, \V{8c}, \V{8d}, \V{9a}, \V{9b}, \V{9c}, \V{10a}, \V{10b}, \V{10c}, \V{11a}, \V{11b}, \V{11c} \\
P27 & \V{1}, \V{3}, \V{7a}, \V{7b}, \V{7c}, \V{8c}, \V{9a}, \V{9b}, \V{9c}, \V{10a}=\V{10b}=\V{10c}, \V{11a}=\V{11b}=\V{11c} \\
P28 & \V{1}, \V{3}, \V{7a}, \V{7b}, \V{7c}, \V{8c}, \V{8d}, \V{9a}, \V{9b}, \V{9c}, \V{10a}, \V{10b}, \V{10c}, \V{11a}, \V{11b}, \V{11c} \\
P29 & \V{1}, \V{3}, \V{7a}, \V{7b}, \V{7c}, \V{8c}, \V{9a}, \V{9b}, \V{9c}, \V{10a}, \V{10b}, \V{10c}, \V{11a}, \V{11b}, \V{11c} \\
P30 & \V{1}, \V{3}, \V{7a}=\V{7b}=\V{7c}, \V{8c}, \V{8d}, \V{9a}, \V{9b}, \V{9c}, \V{10a}=\V{10b}=\V{10c}, \V{11a}=\V{11b}=\V{11c} \\
P31 & \V{1}, \V{3}, \V{7a}, \V{7b}, \V{7c}, \V{8c}, \V{9a}, \V{9b}, \V{9c}, \V{10a}, \V{10b}, \V{10c}, \V{11a}, \V{11b}, \V{11c} \\
P32 & \V{1}, \V{3}, \V{7a}, \V{7b}, \V{7c}, \V{8c}, \V{8d}, \V{9a}, \V{9b}, \V{9c}, \V{10a}, \V{10b}, \V{10c}, \V{11a}, \V{11b}, \V{11c} \\
P33 & \V{1}, \V{3}, \V{7a}, \V{7b}, \V{7c}, \V{8c}, \V{9a}, \V{9b}, \V{9c}, \V{10a}=\V{10b}=\V{10c}, \V{11a}=\V{11b}=\V{11c} \\
P34 & \V{1}, \V{3}, \V{7a}=\V{7b}=\V{7c}, \V{8c}, \V{8d}, \V{9a}, \V{9b}, \V{9c}, \V{10a}, \V{10b}, \V{10c}, \V{11a}, \V{11b}, \V{11c} \\
P35 & \V{1}, \V{3}, \V{7a}, \V{7b}, \V{7c}, \V{8c}, \V{9a}, \V{9b}, \V{9c}, \V{10a}, \V{10b}, \V{10c}, \V{11a}, \V{11b}, \V{11c} \\
P36 & \V{1}, \V{3}, \V{7a}, \V{7b}, \V{7c}, \V{8c}, \V{8d}, \V{9a}, \V{9b}, \V{9c}, \V{10a}=\V{10b}=\V{10c}, \V{11a}=\V{11b}=\V{11c} \\
P37 & \V{1}, \V{7a}=\V{7b}=\V{7c}, \V{8c}, \V{10a}=\V{10b}=\V{10c}, \V{11a}=\V{11b}=\V{11c}, \V{12b}, \V{12c} \\
P38 & \V{1}, \V{7a}, \V{7b}, \V{7c}, \V{8c}, \V{10a}=\V{10b}=\V{10c}, \V{11a}=\V{11b}=\V{11c}, \V{12b}, \V{12c} \\
P39 & \V{1}, \V{7a}, \V{7b}, \V{7c}, \V{8c}, \V{10a}=\V{10b}=\V{10c}, \V{11a}=\V{11b}=\V{11c}, \V{12b}, \V{12c} \\
P40 & \V{1}, \V{7a}, \V{7b}, \V{7c}, \V{8c}, \V{10a}=\V{10b}=\V{10c}, \V{11a}=\V{11b}=\V{11c}, \V{12b}, \V{12c} \\
P41 & \V{1}, \V{7a}=\V{7b}=\V{7c}, \V{8c}, \V{10a}, \V{10b}, \V{10c}, \V{11a}, \V{11b}, \V{11c}, \V{12b}, \V{12c} \\
P42 & \V{1}, \V{7a}, \V{7b}, \V{7c}, \V{8c}, \V{10a}, \V{10b}, \V{10c}, \V{11a}, \V{11b}, \V{11c}, \V{12b}, \V{12c} \\
P43 & \V{1}, \V{7a}, \V{7b}, \V{7c}, \V{8c}, \V{10a}, \V{10b}, \V{10c}, \V{11a}, \V{11b}, \V{11c}, \V{12b}, \V{12c} \\
P44 & \V{1}, \V{7a}, \V{7b}, \V{7c}, \V{8c}, \V{10a}, \V{10b}, \V{10c}, \V{11a}, \V{11b}, \V{11c}, \V{12b}, \V{12c} \\
P45 & \V{1}, \V{7a}=\V{7b}=\V{7c}, \V{8c}, \V{10a}, \V{10b}, \V{10c}, \V{11a}, \V{11b}, \V{11c}, \V{12b}, \V{12c} \\
P46 & \V{1}, \V{7a}, \V{7b}, \V{7c}, \V{8c}, \V{10a}, \V{10b}, \V{10c}, \V{11a}, \V{11b}, \V{11c}, \V{12b}, \V{12c} \\
P47 & \V{1}, \V{7a}, \V{7b}, \V{7c}, \V{8c}, \V{10a}, \V{10b}, \V{10c}, \V{11a}, \V{11b}, \V{11c}, \V{12b}, \V{12c} \\
\multirow{ 2}{*}{\hspace{20pt}P48\hspace{1pt}\normalsize$\quad\left\{\right. $} & \V{1}, \V{7b}, \V{7c}, \V{8c}, \V{10a}, \V{10c}, \V{11a}, \V{11b}, \V{11c}, \V{12b}, \V{12c} \hspace{2.2cm}if $n\bmod 5=0$ and $m=\frac{6n-5}{5}$\\
& \V{1}, \V{7a}, \V{7b}, \V{7c}, \V{8c}, \V{10a}, \V{10b}, \V{10c}, \V{11a}, \V{11b}, \V{11c}, \V{12b}, \V{12c} \hspace{0.9cm}otherwise\\
P49 & \V{1}, \V{7a}=\V{7b}=\V{7c}, \V{8c}, \V{8d}, \V{10a}=\V{10b}=\V{10c}, \V{11a}=\V{11b}=\V{11c}, \V{12b}, \V{12c} \\
P50 & \V{1}, \V{7a}, \V{7b}, \V{7c}, \V{8c}, \V{8d}, \V{10a}=\V{10b}=\V{10c}, \V{11a}=\V{11b}=\V{11c}, \V{12b}, \V{12c} \\
P51 & \V{1}, \V{7a}, \V{7b}, \V{7c}, \V{8c}, \V{8d}, \V{10a}=\V{10b}=\V{10c}, \V{11a}=\V{11b}=\V{11c}, \V{12b}, \V{12c} \\
P52 & \V{1}, \V{7a}, \V{7b}, \V{7c}, \V{8c}, \V{8d}, \V{10a}=\V{10b}=\V{10c}, \V{11a}=\V{11b}=\V{11c}, \V{12b}, \V{12c} \\
P53 & \V{1}, \V{7a}=\V{7b}=\V{7c}, \V{8c}, \V{8d}, \V{10a}, \V{10b}, \V{10c}, \V{11a}, \V{11b}, \V{11c}, \V{12b}, \V{12c} \\
P54 & \V{1}, \V{7a}, \V{7b}, \V{7c}, \V{8c}, \V{8d}, \V{10a}, \V{10b}, \V{10c}, \V{11a}, \V{11b}, \V{11c}, \V{12b}, \V{12c} \\
P55 & \V{1}, \V{7a}, \V{7b}, \V{7c}, \V{8c}, \V{8d}, \V{10a}, \V{10b}, \V{10c}, \V{11a}, \V{11b}, \V{11c}, \V{12b}, \V{12c} \\
P56 & \V{1}, \V{7a}, \V{7b}, \V{7c}, \V{8c}, \V{8d}, \V{10a}, \V{10b}, \V{10c}, \V{11a}, \V{11b}, \V{11c}, \V{12b}, \V{12c} \\
P57 & \V{1}, \V{7a}=\V{7b}=\V{7c}, \V{8c}, \V{8d}, \V{10a}, \V{10b}, \V{10c}, \V{11a}, \V{11b}, \V{11c}, \V{12b}, \V{12c} \\
P58 & \V{1}, \V{7a}, \V{7b}, \V{7c}, \V{8c}, \V{8d}, \V{10a}, \V{10b}, \V{10c}, \V{11a}, \V{11b}, \V{11c}, \V{12b}, \V{12c} \\
P59 & \V{1}, \V{7a}, \V{7b}, \V{7c}, \V{8c}, \V{8d}, \V{10a}, \V{10b}, \V{10c}, \V{11a}, \V{11b}, \V{11c}, \V{12b}, \V{12c} \\
\multirow{ 2}{*}{\hspace{20pt}P60\hspace{1pt}\normalsize$\quad\left\{\right. $} 
& \V{1}, \V{7b}, \V{7c}, \V{8c}, \V{8d}, \V{10a}, \V{10c}, \V{11a}, \V{11b}, \V{11c}, \V{12b}, \V{12c} \hspace{1.5cm}if $n\bmod 5=0$ and $m=\frac{6n-5}{5}$\\
& \V{1}, \V{7a}, \V{7b}, \V{7c}, \V{8c}, \V{8d}, \V{10a}, \V{10b}, \V{10c}, \V{11a}, \V{11b}, \V{11c}, \V{12b}, \V{12c} \hspace{0.2cm}otherwise\\
P61 & \V{1}=\V{7c}, \V{7a}=\V{10a}=\V{10b}=\V{10c}, \V{7b}, \V{8c}, \V{11a}=\V{11b}=\V{11c}, \V{12b}, \V{12c}\\
P62 & \V{1}=\V{7c}=\V{10a}, \V{7b}, \V{8c}, \V{10c}, \V{11a}, \V{11b}, \V{11c}, \V{12b}, \V{12c} \\
P63 & \V{1}=\V{7c}=\V{10a}, \V{7b}, \V{8c}, \V{10c}, \V{11a}, \V{11b}, \V{11c}, \V{12b}, \V{12c} \\
P64 & \V{1}=\V{7c}, \V{7a}=\V{10a}=\V{10b}=\V{10c}, \V{7b}, \V{8c}, \V{8d}, \V{11a}=\V{11b}=\V{11c}, \V{12b}, \V{12c} \\
P65 & \V{1}=\V{7c}=\V{10a}, \V{7b}, \V{8c}, \V{8d}, \V{10c}, \V{11a}, \V{11b}, \V{11c}, \V{12b}, \V{12c} \\
P66 & \V{1}=\V{7c}=\V{10a}, \V{7b}, \V{8c}, \V{8d}, \V{10c}, \V{11a}, \V{11b}, \V{11c}, \V{12b}, \V{12c} \\
P67 & \V{6}, \V{8c}, \V{11a}=\V{11b}=\V{11c}, \V{12b}, \V{12c} \\
P68 & \V{6}, \V{8c}, \V{11a}, \V{11b}, \V{11c}, \V{12b}, \V{12c} \\
P69 & \V{6}, \V{8c}, \V{11a}, \V{11b}, \V{11c}, \V{12b}, \V{12c} \\
P70 & \V{6}, \V{8c}, \V{8d}, \V{11a}=\V{11b}=\V{11c}, \V{12b}, \V{12c} \\
P71 & \V{6}, \V{8c}, \V{8d}, \V{11a}, \V{11b}, \V{11c}, \V{12b}, \V{12c} \\
P72 & \V{6}, \V{8c}, \V{8d}, \V{11a}, \V{11b}, \V{11c}, \V{12b}, \V{12c} \\
P73 & \V{6}, \V{8c}, \V{8d}=\V{11a}=\V{11b}=\V{11c}=\V{12c}, \V{12b} \\
P74 & \V{6}, \V{11c}=\V{12c}, \V{8c}=\V{11b}, \V{11a}, \V{12b} \\
P75 & \V{6}, \V{11c}=\V{12c}, \V{8c}, \V{8d}=\V{11b}, \V{11a}, \V{12b} \\\hline
\end{tabular}
\end{table}

The same analysis can be done for the 74 other polytope types. We give in Table \ref{extreme} the set of extreme points for the 75 distinct polytopes\footnote{These were obtained with the help of a dedicated procedure available at \url{https://github.com/umons-dept-comp-sci/chemichull-paper}.}. For two extreme points V$i$ and V$j$ of a polytope, we write V$i$=V$j$ only if the two points are always equal in the considered polytope. 
When V$i$ and V$j$ are not indicated as being equal, they may coincide for some pairs $(n,m)$, but there is at least one pair $(n,m)$ for which they are not equal. 
For example, for P2, we have \V{11a}=\V{11b}, meaning that these two points are always equal when $n\bmod 12=1$. Still for P2, \V{7b} and \V{11a} are typically distinct, while \V{7b}=\V{11a}(=\V{11b}=\V{11c})=(5,0,2) for $n=13$.

\begin{table}
	\scriptsize	\centering
	\caption{Extreme points for $\mathcal{P}_{n, m}$ when $(n,m)$ satisfy~\eqref{m>12} and $n$ is even}\label{Tab:synthese}~\\
	\setlength{\tabcolsep}{1pt}
	\begin{tabular}{|cccccccccccccccccccc|ll}
\multicolumn{1}{c}{\V{1}}&\V{2}&\V{3}&\V{6}&\V{7a}&\V{7b}&\V{7c}&\V{8c}&\V{9a}&\V{9b}&\V{9c}&\V{10a}&\V{10b}&\V{10c}&\V{11a}&\V{11b}&\V{11c}&\V{12a}&\V{12b}&\multicolumn{1}{c}{\V{12c}}&&\\\cline{1-20}
&x&&&x&x&x&x&&&&x&x&x&x&x&x&x&&&\multicolumn{2}{l}{$\quad m=n-1$ (chemical trees)}\\
x&&&&x&x&x&x&&&&x&x&x&x&x&x&&&&\multicolumn{2}{l}{$\quad m=n\hspace{0.5cm}$ (unicyclic chemical graphs)}\\
x&&x&&x&x&x&x&x&x&x&x&x&x&x&x&x&&&&\multicolumn{2}{l}{$\quad m=n+1$ (bicyclic chemical graphs)}\\
x&&&&x&x&x&x&&&&x&x&x&x&x&x&&x&x&$\quad n + 1<m<\frac{6n}{5}$ and &$6n{-}5m{\notin}\{1,2,5\}$\\
x&&&&&x&x&x&&&&x&&x&x&x&x&&x&x&$\quad n + 1<m<\frac{6n}{5}$ and &$6n{-}5m{\in}\{1,2,5\}$\\
&&&x&&&&x&&&&&&&x&x&x&&x&x&$\quad\frac{6n}{5} \le m\leq \left\lfloor\frac{3n-3}{2} \right\rfloor$&\\\cline{1-20}
\end{tabular}\end{table}
As indicated in Table \ref{polytopes}, polytopes P1–P12 relate to chemical trees since they are defined for  
$m=n-1$. It is interesting to observe that if 
$n$ is even, the six polytopes P1, P3, P5, P7, P9, and P11 share the same set of extreme points, namely: \V{2}, \V{7a}, \V{7b}, \V{7c}, \V{8c}, \V{10a}, \V{10b}, \V{10c}, \V{11a}, \V{11b}, \V{11c}, and  \V{12a}. Some of these points are sometimes equal to each other. If 
$n$ is odd, point \V{8d} is added to this list. By performing a similar analysis for other values of 
$m$, we obtain Table \ref{Tab:synthese}, which shows the extreme points for an even number of vertices. If the number of vertices is odd, point \V{8d}  should be added to the lists in Table \ref{Tab:synthese}.

Let's now analyze the pairs $(n,m)$ that satisfy condition (\ref{val_n_m}), but not condition (\ref{m>12}). The cases where $n\geq 6$ and $m=\frac{3n-2}{2}$ or when $n\ge 4$ and $m=\lfloor\frac{3n}{2}\rfloor$ are treated in Section \ref{sec:Degenerated}. The other cases which all have $m\leq 11$ can easily be analyzed by generating all chemical graphs of order $n$ and size $m$, using for example \emph{Nauty's geng}~\cite{geng}, and then building the polytope $\mathcal{P}_{n,m}$ for determining its extreme points. It is therefore not difficult to determine the extreme points for all these pairs $(n,m)$. We list in Table \ref{m<12} the  extreme points for these 32 particular cases. The third column of the table indicates the dimension of $\mathcal{P}_{n,m}$.
We observe that among these 32 polytopes, 5 are of dimension 0, 3 of dimension 1, 2 of dimension 2 and 22 of dimension 3. Among the 22 full-dimensional polytopes, 11 are special cases of the polytopes of Table \ref{polytopes}. We indicate their Id in the last column of Table \ref{m<12}. The 11 new full-dimensional polytopes are named P76-P86, and their Id appears in italics in the last column of Table \ref{m<12}.

\begin{table}[!htb]
\scriptsize\centering
\caption{Extreme points of $\mathcal{P}_{n,m}$ for pairs $(n,m)$ satisfying (\ref{val_n_m}) but not (\ref{m>12})}\label{m<12}~\\[1ex]
\begin{tabular}{|c|c|c|l|c|} 
\cline{1-5}
\multicolumn{1}{|c|}{$n$} & \multicolumn{1}{c}{$m$} & \multicolumn{1}{|c}{dim.} & \multicolumn{1}{|c|}{extreme points} &\multicolumn{1}{c|}{Id}\\ \cline{1-5}
$3$ & $2$ & $0$ & (2,0,0)&\\ 
$3$ & $3$ & $0$ & (0,0,0)&\\ 
$4$ & $3$ & $1$ & (0,3,0),\;(2,0,0)&\\ 
$4$ & $4$ & $1$ & (0,0,0), (0,1,0)&\\ 
$4$ & $5$ & $0$ & (0,0,1)&\\ 
$5$ & $4$ & $1$ & (1,2,0), (2,0,0)&\\ 
$5$ & $5$ & $3$ & (0,0,0), (0,1,0), (0,2,1), (1,0,0)&\it{P76}\\ 
$5$ & $6$ & $2$ & (0,0,0), (0,0,1), (0,1,3)&\\ 
$6$ & $5$ & $3$ & (0,4,1), (1,2,0), (2,0,0), (2,1,0)&\it{P77}\\ 
$6$ & $6$ & $3$ & (0,0,0), (0,2,0), (0,3,3), (1,0,0), (1,1,1)&\it{P78}\\ 
$6$ & $7$ & $3$ & (0,0,0), (0,0,1), (0,1,2), (0,2,5), (1,0,3)&\it{P79}\\ 
$7$ & $6$ & $3$ & (0,4,0), (1,3,1), (2,0,0), (3,0,0)&\it{P80}\\ 
$7$ & $7$ & $3$ & (0,0,0), (0,2,0), (0,3,2), (1,0,0), (1,1,0), (1,2,3), (2,0,1)&\it{P81}\\ 
$7$ & $8$ & $3$ & (0,0,0), (0,0,1), (0,1,1), (0,1,3), (0,2,4), (1,0,2), (1,0,3), (1,1,5)&\it{P82}\\ 
$7$ & $9$ & $3$ & (0,0,3), (0,0,5), (0,1,6), (1,0,7)&P73\\ 
$8$ & $7$ & $3$ & (0,4,0), (0,5,2), (1,3,0), (2,0,0), (2,2,1), (3,0,0)&\it{P83}\\ 
$8$ & $8$ & $3$ & (0,0,0), (0,2,0), (0,3,1), (0,4,4), (2,0,0), (2,0,1), (2,1,3)&\it{P84}\\ 
$8$ & $9$ & $3$ & (0,0,0), (0,0,1), (0,1,0), (0,1,3), (0,3,6), (1,0,1), (1,1,5), (2,0,5)&\it{P85}\\ 
$8$ & $10$ & $3$ & (0,0,2), (0,0,5), (0,2,8), (1,0,6), (1,0,7)&P74\\ 
$9$ & $8$ & $3$ & (0,4,0), (0,5,1), (1,4,2), (2,0,0), (2,2,0), (3,0,0), (3,1,1)&\it{P86}\\ 
$9$ & $9$ & $3$ & (0,0,0), (0,3,0), (0,4,3), (1,3,4), (2,0,0), (3,0,3)&P22\\ 
$9$ & $10$ & $3$ & (0,0,0), (0,0,1), (0,1,0), (0,1,3), (0,2,2), (0,3,5), (1,0,0), (1,1,5), (1,2,6), (2,0,4), (2,0,5)&P34\\ 
$9$ & $11$ & $3$ & (0,0,1), (0,0,5), (0,2,7), (1,0,5), (1,0,7), (1,1,8)&P75\\ 
$10$ & $9$ & $3$ & (0,4,0), (0,5,0), (0,6,3), (2,0,0), (3,0,0), (3,1,0), (4,0,1)&P11\\ 
$10$ & $10$ & $3$ & (0,0,0), (0,3,0), (0,4,2), (0,5,5), (1,2,0), (2,0,0), (2,2,4), (3,0,2), (3,0,3)&P23\\ 
$10$ & $11$ & $3$ & (0,0,0), (0,0,1), (0,1,0), (0,1,3), (0,2,1), (0,4,7), (1,0,0), (1,1,5), (2,0,3), (2,0,5), (2,1,6)&P35\\ 
$11$ & $10$ & $3$ & (0,4,0), (0,5,0), (0,6,2), (1,4,0), (1,5,3), (2,0,0), (3,2,2), (4,0,0), (4,0,1)&P12\\ 
$11$ & $11$ & $3$ & (0,0,0), (0,3,0), (0,4,1), (0,5,4), (1,4,5), (2,0,0), (2,1,0), (3,0,1), (3,0,3), (3,1,4)&P24\\ 
$12$ & $11$ & $3$ & (0,4,0), (0,5,0), (0,6,1), (0,7,4), (2,0,0), (2,3,0), (4,0,0), (4,0,1), (4,1,2)&P1\\ 
even $\ge 4$ & $\frac{3n}{2}$ & $0$ & (0,0,$m$)& \\ 
odd $ \ge 5$ & $\frac{3n-1}{2}$ & $0$ & (0,0,$m-$2)& \\ 
even $\ge 6$ & $\frac{3n-2}{2}$ & $2$ & (0,0,$m-$4), (0,0,$m-$3), (0,1,$m-$1)& \\ \cline{1-5}
\end{tabular}
\end{table}

In summary, our polyhedral description is based on a partition of the set of all possible combinations of $n$ and $m$ into 96 cases: 5 polytopes are of dimension 0, 3 of dimension 1, 2 of dimension 2 and 86 of dimension 3.
The facets of the degenerated polytopes are given in Section \ref{sec:Degenerated}. For completeness, we give in Table \ref{tab_11polytopes} the facets of the 3-dimensional polytopes P76-P86, where F22, F23, F24 and F25 are new facets defined as follows:
\begin{itemize}[nosep]
\item  F22: $2m_{13}-4m_{33}\ge 0$;
\item F23: $4m_{13}-4m_{33}\ge 0$;\item F24: $2m_{12}+5m_{13}-4m_{33}\ge 0$;
\item F25: $2m_{12}+8m_{13}-4m_{33}\ge 0$.
\end{itemize}
\begin{table}[H]
	\scriptsize\centering
\caption{Facets of the full-dimensional polytopes P76-P86}\label{tab_11polytopes}~\\[2ex]
\begin{tabular}{|c|l||c|l|}
\hline Id&facets&Id&facets\\
\hline
P76&\G{2}, \G{5}, \G{11}, F22&P82&\G{1}, \G{2}, \G{3}, \G{6}, \G{9}, \G{14}, \G{16}\\
P77&\G{5}, \G{11}, \G{15}, \G{17}&P83&\G{1}, \G{5}, \G{10}, \G{12}, \G{15}, \G{17}, \G{20}\\
P78&\G{1}, \G{2}, \G{5}, \G{14}, F23&P84&\G{1}, \G{2}, \G{3}, \G{4}, \G{5}, \G{9}, \G{11}, F25\\
P79&\G{1}, \G{2}, \G{3}, \G{7}, \G{12}&P85&\G{1}, \G{2}, \G{3}, \G{4}, \G{7}, \G{13}, \G{16}\\
P80&\G{1}, \G{5}, \G{6}, \G{15}&P86&\G{1}, \G{5}, \G{6}, \G{11}, \G{14}, \G{15}, \G{19}\\
P81&\G{1}, \G{2}, \G{3}, \G{5}, \G{10}, \G{13}, \G{18}, F24&&\\\hline
\end{tabular}
\end{table}

\section{Example of application}\label{sec:application}
In this section, we illustrate how the results of this paper can be applied to determine chemical graphs that maximize or minimize a degree-based topological index. Other examples are given in \cite{ChemicHullPaper}, where we show that many published results can be easily reproduced in a few lines using the polyhedral approach presented in this paper.

Consider the problem of minimizing the Albertson topological index \cite{Albertson} with $c_{ij}=\mid i-j\mid$. As mentioned in Section \ref{sec:basic}, minimizing $c_{12}m_{12}+c_{13}m_{13}+c_{22}m_{22}+c_{23}m_{23}+c_{33}m_{33}$ is equivalent to minimizing $f(m_{12},m_{13},m_{33})=c'_{12}m_{12}+c'_{13}m_{13}+c'_{33}m_{33}$, where
$c'_{12}=c_{12}-4c_{22}+3c_{23}$,
$c'_{13}=c_{13}-3c_{22}+2c_{23}$ and
$c'_{33}=c_{22}-2c_{23}+c_{33}.$
For the Albertson index, this means that we have to minimize $f(m_{12},m_{13},m_{33})=4m_{12}+4m_{13}-2m_{33}$.

If we are interested, for example, in graphs of order $n=13$ and size $m=15$, then we have to consider P64, and as shown in Table~\ref{extreme}, the set of extreme points is $\{$\V{1}{=}\V{7c}, \V{7a}{=}\V{10a}{=}\V{10b}{=}\V{10c}, \V{7b}, \V{8c}, \V{8d}, \V{11a}{=}\V{11b}{=}\V{11c}, \V{12b}, \V{12c}$\}$. Since $f($\V{1})=0, $f($\V{7a})=4, $f($\V{7b})=2, $f($\V{8c})=-2, $f($\V{8d})=-4, $f($\V{11a})=-6 and $f($\V{12b})=$f($\V{12c})=-10, we deduce that the points that minimize the Albertson index are \V{12b}=(0,0,5) and \V{12c}=(1,0,7). Examples of chemical graphs that minimize the Albertson index for $(n,m){=}(13,15)$ are shown in Figure \ref{figure4}.

If we are interested in minimizing the Albertson index for all possible pairs $(n,m)$ that satisfy condition (\ref{m>12}), then we can compute the value $f($V$i)$ for all $i=1,\ldots,21$. These values are given in Table \ref{albertson}.

\begin{figure}[!htb]
	\centering\includegraphics[scale=0.75]{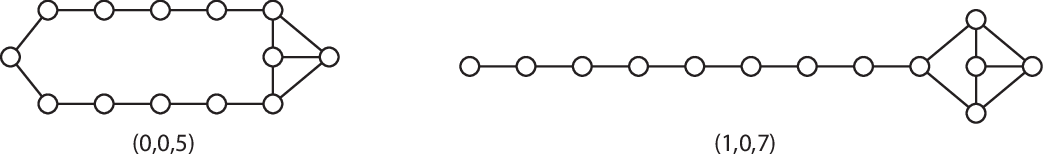}
	\vspace{-0.2cm}\caption{Two chemical graphs that minimize the Albertson index for $(n,m)=(13,15)$}
	\label{figure4}
\end{figure}

\begin{table}[h!]
	\footnotesize
\centering
	\caption{$f($V$i)$ $(i=1,\ldots,23)$ for the Albertson index}\label{albertson}~\\[2ex]
	\begin{tabular}{|c|c||c|c||c|c|} \hline	
    V$i$&$f($V$i)$&V$i$&$f($V$i)$&V$i$&$f($V$i)$\\\hline
    \V{1}&0&                        \V{8c}&$9n-8m+n\bmod 2$&                    \V{10c}&$\frac{4}{3}(6n-5m-\frac{(2m)\bmod 3}{2})$\\
    \V{2}&8&                        \V{8d}&$9n-8m-1$&                           \V{11a}&$\frac{2}{3}(12n-11m+2(m\bmod 3))$\\
    \V{3}&-2&                       \V{9a}&$12m-12n-14$&                        \V{11b}&$\frac{2}{3}(12n-11m+m\bmod 3)$\\
    \V{6}&$12n-10m$&                \V{9b}&$-2$&                                \V{11c}&$\frac{2}{3}(12n-11m-m\bmod 3)$\\
    \V{7a}&$6n-5m+(m-2n)\bmod 4$&   \V{9c}&$-2$&                                \V{12a}&$12n-12m+4$\\
    \V{7b}&$6n-5m-(m-2n)\bmod 4$&   \V{10a}&$\frac{4}{3}(6n-5m-m\bmod 3)$&      \V{12b}&$6n-6m+2$\\
    \V{7c}&$6n-5m-(2n-m)\bmod 4$&   \V{10b}&$\frac{4}{3}(6n-5m-m\bmod 3)$&      \V{12c}&$6n-6m+2$\\\hline
    \end{tabular}
\end{table}

It is not difficult to check that the minimum value of $f($V$i)$ is $12n-12m+4$ for $m\ge n+1$, but the unique point that reaches this value is \V{12a} which is an extreme point only if $m=n-1$, while for $m=n-1$, a better value is obtained, for example by \V{1}. The second best value is $6n-6m+2$ which is reached by \V{12b} and \V{12c} for $m\ge n+1$. Since \V{12b} and \V{12c} are extreme points only if $m\ge n+2$, we deduce that these two points minimize the Albertson index for $m\ge n+2$. The third best value is -2 which is reached by \V{3}, \V{9a}, \V{9b} and \V{9c} when $m=n+1$. The fourth best value is 0 and is reached by \V{1} for $m=n$. The fifth best value is 8 which is reached by \V{2} for $m=n-1$. Note that when there are more than one extreme point for a pair $(n,m)$, then, according to Theorem 3.1 of \cite{Hansen2017}, all integer convex combinations of extreme points are realizable and therefore also optimal. This occurs here when $m=n+1$. Indeed, we have $(1,0,3)=\frac{1}{2}$(\V{3}+\V{9b}).
In summary, the best extreme points are those given in Table \ref{minAlbertson}. This is exactly the same result as what is stated in \cite{article33}. Note that once we know the list of extreme points that minimize the Albertson index (i.e. the points in Table \ref{minAlbertson}), it is very easy to determine which ones are optimal for a given pair $(n,m)$. For the above example with $(n,m)=(13,15)$, we have $m\ge n+2$, which implies that \V{12b}=(0,0,5) and \V{12c}=(1,0,7) are the points we are interested in.

\begin{table}[h!]
	\footnotesize
\centering
	\caption{Points that minimize the Albertson index for $(n,m)$ satisfying condition (\ref{m>12})}\label{minAlbertson}~\\
		\begin{tabular}{r| c | c |c | l}\cline{2-4} 
			&$m_{12}$ & $m_{13}$ & $m_{33}$&\\ \cline{2-4}
			\V{2}&2 & 0 & 0& if $m=n-1$\\ \cline{2-4}
			\V{1}&0 & 0 & 0& if $m=n$\\ \cline{2-4}
			\V{3}&0 & 0 & 1& \multirow{5}{*}{if $m=n+1$}\\			
				\V{9a}&1 & 1 & 5& \\
	\V{9b}&2 & 0 & 5& \\ 
			\V{9c}&0 & 1 & 3& \\
            $\frac{1}{2}$(\V{3}+\V{9b})&1 & 0 & 3&\\\cline{2-4}
			\V{12b}&0 & 0 & $3m-3n-1$& \multirow{2}{*}{if $n+2\le m\leq \frac{3n-3}{2}$}\\
			\V{12c}&1 & 0 & $3m-3n+1$&\\			 \cline{2-4}
		\end{tabular}
\end{table}

\section{Concluding remarks and future work}\label{sec:futurework}

We have given a complete polyhedral description of all chemical graphs of maximum degree at most 3. This allowed us to describe the convex hull associated with chemical graphs for any pair $(n, m)$, and to observe that the maximum number of extreme points is 16. This is an important reinforcement of the formulation induced by Hansen \emph{et al.}~\cite{Hansen2017}. Indeed, as they mention in their paper, Equations \eqref{Hansen1}{-}\eqref{Hansen7} can be transformed into an integer linear program (IP) with $2^{\Delta - 1} + \Delta(\Delta + 2) - 2$ variables and $2^\Delta + \frac{\Delta}{2}
 (\Delta^2 - 
\Delta + 6) - 4 + B_{\Delta}$ constraints, where $\Delta$ is an upper bound on the maximum degree and $B_i$ is the $i^{\text{th}}$ Bell number.  For example, if $n = 14$, $m = 13$ and $\Delta=3$, the IP induced by  Equations \eqref{Hansen1}{-}\eqref{Hansen7} has 17 variables and 27 constraints. Its linear relaxation is very far from the convex hull defined in this paper (12 extreme points and 10 facets) since it contains $11\,648$ extreme points.

The same work could be done for a maximum degree at most 4. The work would be essentially the same. More precisely, a chemical graph would then be characterized using 9 variables instead of 5, since $m_{14}$, $m_{24}$, $m_{34}$ and $m_{44}$ are then potentially strictly positive. Here too, when the values of $n$ and $m$ are fixed, this reduces the degree of freedom by 2 and we can therefore work in a 7-dimensional space spanned by 7 of the 9 variables. To obtain a polyhedral description, it would be necessary, as we did for a maximum degree at most 3, to determine inequalities that are valid when conditions on $n$ and $m$ are imposed, and to show that they are facet defining by exhibiting 7 affinely independent points of the hyperplane defined by each inequality. Subsequently, for a pair $(n,m)$, by denoting $\mathcal{P}'_{n,m}$ the polytope defined by the set of active facets for $(n,m)$, we could examine all the intersections of 7 facets and show that if the intersection exists and is in $\mathcal{P}'_{n,m}$, then it corresponds to a  point that is realizable for $(n,m)$. 

The work to be done for a complete polyhedral description of chemical graphs of maximum degree at most 4 is however much more complex than it was for a maximum degree at most 3. Indeed, while the polytopes that we have described in this paper all have at most 10 facets and at most 16 extreme points, these numbers explode when the maximum degree can be equal to 4. For example, we generated all the chemical graphs of maximum degree at most 4 for $n=21$ and $m=23$ and we were thus able to observe that the associated polytope contains 761 facets and 402 extreme points. In summary, although this complete polyhedral description is possible, using the same tools as those described in this article, the number of cases to be considered is enormous and therefore requires considerable work that we nevertheless plan to carry out in the near future. 

We conclude by pointing out that the ChemicHull website~\cite{ChemicHull} mentioned in Section~\ref{sec:realizable} allows, for any given pair $(n, m)$, to list and visualize the associated polytope, its facets and its extreme points.

\bibliographystyle{acm}
\bibliography{main}

\end{document}